\documentclass[11pt]{article}

\usepackage{amsmath}
\usepackage{amssymb}
\usepackage{amscd}
\usepackage{amsthm}
\usepackage{indentfirst}
\usepackage{tikz}

\usepackage{hyperref}
\hypersetup{colorlinks=true,linkcolor=red,pdfpagemode=UseNone,pdfstartview=FitH}

\renewcommand\labelenumi{(\roman{enumi})}
\renewcommand\theenumi\labelenumi

\newcommand{\Z}{\ensuremath{\mathbb{Z}}}
\newcommand{\Q}{\ensuremath{\mathbb{Q}}}

\newcommand{\F}{\ensuremath{\mathbb{F}}}

\usepackage[margin=2.6cm]{geometry}

\theoremstyle{plain}
\newtheorem{thm}{Theorem}[section]
\newtheorem{lem}[thm]{Lemma}
\newtheorem{cor}[thm]{Corollary}
\newtheorem{prop}[thm]{Proposition}

\theoremstyle{definition}
\newtheorem{rmq}[thm]{Remark}

\DeclareMathOperator{\Pic}{Pic}
\DeclareMathOperator{\Div}{Div}
\DeclareMathOperator{\PDiv}{PDiv}
\DeclareMathOperator{\ord}{ord}
\DeclareMathOperator{\divisor}{div}
\DeclareMathOperator{\dis}{disc}
\DeclareMathOperator{\Gal}{Gal}

\DeclareMathOperator{\rk}{rk}

\DeclareMathOperator{\contr}{contr}

\newcommand{\PP}{\mathbb{P}}
\newcommand{\Egp}{\mathcal{N}}
\newcommand{\Emin}{\mathcal{E}}
\newcommand{\Gm}{\mathbf{G}_{{\rm m}}}
\newcommand{\tors}{\text{\rm tors}}
\newcommand{\kbar}{\bar{k}}

\def \O {\mathcal{O}}

\begin{document}

\title{Integral points on elliptic curves with $j$-invariant $0$ over $k(t)$}

\author{Jean Gillibert \and Emmanuel Hallouin \and Aaron Levin}

\date{January 2024}

\maketitle

\begin{abstract}
We consider elliptic curves defined by an equation of the form $y^2=x^3+f(t)$, where $f\in k[t]$ has coefficients in a perfect field $k$ of characteristic not $2$ or $3$. By performing $2$ and $3$-descent, we obtain, under suitable assumptions on the factorization of $f$, bounds for the number of integral points on these curves. These bounds improve on a general result by Hindry and Silverman. When $f$ has degree at most $6$, we give exact expressions for the number of integral points of small height in terms of certain subgroups of Picard groups of the $k$-curves corresponding to the $2$ and $3$-torsion of our curve. This allows us to recover explicit results by Bremner, and gives new insight into Pillai's equation.
\end{abstract}


\section{Introduction}

Let $k$ be a perfect field of characteristic not $2$ or $3$. In the applications we have in mind, $k$ may be a number field, or a finite field, or the algebraic closure of such fields.

In this paper, we consider elliptic curves $E$ over $k(t)$ defined by an equation of the form
$$
y^2=x^3+f(t),
$$
where $f\in k[t]$ is a polynomial which we require to satisfy specific assumptions depending on the statements we prove. More precisely, we are interested in the set of integral points on (the given Weierstrass equation of) $E$, that is, the set
\begin{align*}
E(k[t]):=\{(x,y)\in E(k(t))\mid x,y\in k[t]\}.
\end{align*}
When $k$ has characteristic $0$ and $f$ is not a 6th power in $k[t]$ (i.e., $E$ is nonconstant), it was proved by Lang \cite{lang1960} that this set is finite; an effective proof was given by Mason \cite[Chap.~IV, \S{}3]{mason84}. In positive characteristic, this set may be infinite for specific $f$, see \cite[Theorem~1.2]{SS}.

While the elliptic curves $y^2=x^3+f(t)$ are quite special, the integral points on such curves gain increased importance from the observation that they are essentially in bijection with elliptic curves over $k(t)$ with discriminant $f(t)$, up to a constant (see \cite[\S 2.1]{Shioda}).  This idea was famously exploited earlier by Shafarevich over number fields \cite{shafarevich62}.

In continuation of previous work by the first and the third author \cite{gl18a}, we perform on such curves explicit $2$ and $3$-descent computations, in close analogy with the number field case. Under specific requirements on $f$, we prove that the $2$ and $3$-descent maps are finite-to-one when restricted to $E(k[t])$. This allows us to deduce explicit upper bounds for the number of integral points on $E$. Our main result is the following.

\begin{thm}
\label{intbound}
Let $k$ be a perfect field of characteristic not $2$ or $3$, and let $E$ be an elliptic curve over $k(t)$ defined by the Weierstrass equation $y^2=x^3+f(t)$, for some $f\in k[t]$.

Assume that $f=f_1f_2^2f_3^3f_4^4f_5^5$ where the $f_i$ are pairwise coprime, separable polynomials in $k[t]$, and $f_1$ is nonconstant. We let $d:=\deg f$, $d_i:=\deg f_i$ and we denote by $\omega_i$ the number of irreducible factors of $f_i$.

\begin{enumerate}
\item Assume that $3\nmid d$, and that $f=f_1f_2^2f_4^4$. Then the set~$E(k[t])$ of integral points on~$E$ satisfies
\begin{align*}
\#E(k[t])&\leq 2^{\rk_{\Z} E(k(t))+1}(2^{\omega_2+\omega_4+1}+1)-2\leq 2^{2d_1+2d_2-1}\left(2^{\omega_2+\omega_4+1}+1\right)-2.
\end{align*}
If in addition $f=f_1$, the previous bound can be sharpened as follows
\begin{align*}
\#E(k[t])&\leq 2^{\rk_{\Z} E(k(t))+2}-2\leq 2^{2d}-2.
\end{align*}
\item Assume $d$ is odd, and that $f=f_1f_3^3f_5^5$. Then the set~$E(k[t])$ of integral points on~$E$ satisfies
\begin{align*}
\#E(k[t])&\leq
 \begin{cases}
 2^{\omega_3}\cdot \left(3^{\rk_{\Z} E(k(t))+1}-1\right) & \text{if $\zeta_3\notin k$} \\
 2^{\omega_3}\cdot \left(3^{\frac{1}{2}\rk_{\Z} E(k(t))+2} -3\right) & \text{if $\zeta_3\in k$}
 \end{cases}
\end{align*}
where $\zeta_3$ denotes a primitive third root of unity in $\kbar$. In particular,
\begin{align*}
\#E(k[t])\leq 2^{\omega_3}\cdot (3^{d+1}-3).
\end{align*}
\end{enumerate}
\end{thm}

The first item is obtained by performing $2$-descent, while the second uses $3$-descent. The statements of Theorem~\ref{intbound} are extracted from \S{}\ref{subsect:2descent}, Theorem~\ref{thm:main},  and \S{}\ref{subsect:3descent}, Theorem~\ref{thm:main-3}.

For small values of $d$, the bounds above are close to being optimal.  For instance, when $f$ is square-free and quadratic ($d=2$), the bound in Theorem~\ref{intbound} (i) asserts that $\#E(k[t])\leq 14$, and an easy argument (Remark~\ref{d2rem}) refines this bound to $\#E(k[t])\leq 12$.  On the other hand, if $k$ is algebraically closed of characteristic $0$, then it is known \cite[Th.~8.2]{Shioda} that $\#E(k[t])=12$.

If $K$ is a one-dimensional function field of characteristic $0$, and if $T$ is a finite set of places of $K$, then Hindry and Silverman \cite{HS} proved a function field version of a conjecture of Lang, giving a bound for the number of $T$-integral points on any non-constant elliptic curve $E$ over $K$ given by a $T$-minimal Weierstrass equation.  Their bound depends only on $\# T$, the genus of $K$, and the rank of $E(K)$.  In the situation of Theorem~\ref{intbound}, their result takes the form
\begin{align*}
\#E(k[t])\leq 144\cdot 10^{7.1 \rk_{\Z} E(k(t))}.
\end{align*}
Thus, in our restricted setting, Theorem~\ref{intbound} yields a notable quantitative improvement to the result of Hindry-Silverman (and extends it to positive characteristic $\neq 2,3$).  We note that the method used here is completely different from the technique used by Hindry and Silverman, which relied on a careful study of canonical heights.

In a complementary direction, when $d:=\deg f\leq 6$,  we study integral points of small height on $E$ in terms of the arithmetic of the curves $C_2$ and $C_3$ arising from the $2$ and $3$-descent computations for $E$ (see Theorem~\ref{corSh2} below). It is well-known that $d\leq 6$ is equivalent to the fact that $E$ yields a rational elliptic surface \cite[\S{}10]{Shioda2}.
In that case, the arithmetic genus $\chi$ of N{\'e}ron's minimal regular model $\Emin$ of $E$ is equal to~$1$.  We say that an integral point $P\in E(k[t])$ is \emph{exceptional} if $\hat{h}(P)>\chi=1$, where $\hat{h}$ is the canonical height on $E$. The terminology is justified by a result of Shioda \cite{Shioda} that for generic $f$ of degree $d\leq 6$, $E(k[t])$ does not contain any exceptional integral points.

For $f\in k[t]$ nonconstant and sixth-power free of degree $d\leq 6$, let $F(t,u)=u^6f(t/u)$ and let $F(t,u)=F_1F_2^2F_3^3F_4^4F_5^5$ where the $F_i$ are pairwise coprime separable binary forms of degree $\delta_i$.  We call the quintuple $(\delta_1,\ldots, \delta_5)$ the type of $f$. If $f$ has type $(\delta_1,\ldots, \delta_5)$, then $\Emin$ has $\delta_1,\ldots, \delta_5$ singular fibers of types $\mathrm{II}$, $\mathrm{IV}$, $\mathrm{I}_0^*$, $\mathrm{IV}^*$, $\mathrm{II}^*$, respectively (see Lemma~\ref{lem:badE}).

For each possible type of $f$, we give information on the number of nonexceptional integral points in $E(\kbar[t])$ with a $k(t)$-rational ($x$- or $y$-)coordinate.  In many cases, we give exact expressions in terms of the number of rational torsion points in certain subsets of an associated Picard group (see Section \ref{sec:EPic} for the definition of the Picard groups used).

\begin{thm}
\label{corSh2}
Let $k$ be a perfect field of characteristic not $2$ or $3$, and let $E$ be an elliptic curve over $k(t)$ defined by the Weierstrass equation $y^2=x^3+f(t)$, where $f\in k[t]$ is nonconstant of degree at most $6$.  If $f$ is not a perfect power in $\kbar[t]$, let $C_2$, $C_3$, and $C_3'$ be the smooth projective $k$-curves defined by the affine equations
\begin{align*}
C_2: x^3&=-f(t),\\
C_3: y^2&=f(t),\\
C_3': y^2&=-27f(t).
\end{align*}
\begin{enumerate}
\item If $f$ is of type $(1,0,0,0,1)$ then $E(\kbar(t))$ is trivial.
\item If $f$ is of type $(0,1,0,1,0)$ then $E(\kbar(t))=\{\infty,\pm (0,\sqrt{f})\}\cong\mathbb{Z}/3\mathbb{Z}$ for some choice of $\sqrt{f}\in \kbar[t]$.
\item If $f$ is of type $(2,0,0,1,0)$, then
\begin{align*}
\#\left\{P\in E(\kbar(t))\mid x(P)\in k[t], \hat{h}(P)=1/3\right\}&=\#\left\{P\in E(\kbar(t))\mid x(P)\in k[t], \hat{h}(P)=1\right\}\\
&=2(\#\Pic(C_2)[2]-1),\\
\#\left\{P\in E(\kbar(t))\mid y(P)\in k[t], \hat{h}(P)=1/3\right\}&=3(\#\Pic(C_3,\mathbb{Q}.D_3)[3]-1),\\
\#\left\{P\in E(\kbar(t))\mid y(P)\in k[t], \hat{h}(P)=1\right\}&=3(\#\Pic(C_3',\mathbb{Q}.D_3')[3]-1).
\end{align*}
\item If $f$ is of type $(0,0,2,0,0)$, then $E(\kbar(t))=\{\infty, (\sqrt[3]{f},0), (\zeta_3 \sqrt[3]{f},0), (\zeta_3^2 \sqrt[3]{f},0)\}\cong \mathbb{Z}/2\mathbb{Z}\oplus \mathbb{Z}/2\mathbb{Z}$, for some choice of $\sqrt[3]{f}\in \kbar[t]$ and $\zeta_3$ a primitive third root of unity .
\item If $f$ is of type $(1,1,1,0,0)$, then
\begin{align*}
\#\left\{P\in E(\kbar(t))\mid x(P)\in k[t], \hat{h}(P)=1/6\right\}&=\#\left\{P\in E(\kbar(t))\mid x(P)\in k[t], \hat{h}(P)=1/2\right\}\\
&=2(\#\Pic(C_2,\mathbb{Q}.D_2)[2]-1),\\
\#\left\{P\in E(\kbar(t))\mid y(P)\in k[t], \hat{h}(P)=1/6\right\}&=3(\#\Pic(C_3,\mathbb{Q}.D_3)[3]-1),\\
\#\left\{P\in E(\kbar(t))\mid y(P)\in k[t], \hat{h}(P)=1/2\right\}&=3(\#\Pic(C_3',\mathbb{Q}.D_3')[3]-1).
\end{align*}

\item If $f$ is of type $(3,0,1,0,0)$, then
\begin{align*}
\#\{P\in E(\kbar(t))\mid x(P)\in k[t], \hat{h}(P)=1/2\}&=2(\#\Pic(C_2,\mathbb{Q}.D_2)[2]-\#\Pic(C_2)[2]),\\
\#\{P\in E(\kbar(t))\mid y(P)\in k[t], \hat{h}(P)=1/2\}&=\#\{P\in E(\kbar(t))\mid y(P)\in k[t], \hat{h}(P)=1\}\\
&=3(\#\Pic(C_3)[3]-1),\\
\#\{P\in E(\kbar(t))\mid x(P)\in k[t], \hat{h}(P)=1\}&=2(\#\Pic(C_2)[2]-1)\cdot\\
&(\#\text{nontrivial $k$-rational cyclic subgroups of } \Pic(C_3)(\kbar)[3]).
\end{align*}
Furthermore,
\begin{align*}
\#\{P\in E(k[t])\mid \hat{h}(P)=1\}&=(\#\Pic(C_2)[2]-1)(\#\Pic(C_3)[3]-1).
\end{align*}
\label{1.2vi}
\item If $f$ is of type $(0,3,0,0,0)$ then
\begin{align*}
\#\left\{P\in E(\kbar(t))\mid x(P)\in k[t], \hat{h}(P)=1/3\right\}&=\#\left\{P\in E(\kbar(t))\mid x(P)\in k[t], \hat{h}(P)=1\right\}\\
&=2\#\{P\in \mathbb{P}^1(k)\mid F(P)=0\}\cdot \#\{\alpha\in k\mid \alpha^3=2a^2\Delta\},\\
\#\left\{P\in E(\kbar(t))\mid y(P)\in k[t], \hat{h}(P)=1/3\right\}&=3\#\{P\in \mathbb{P}^1(k)\mid F(P)=0\}\cdot \#\{\alpha\in k\mid \alpha^2=a\Delta\}\\
\#\left\{P\in E(\kbar(t))\mid y(P)\in k[t], \hat{h}(P)=1\right\}&=3\#\{P\in \mathbb{P}^1(k)\mid F(P)=0\}\cdot \#\{\alpha\in k\mid \alpha^2=-3a\Delta\}.
\end{align*}
where $a$ is the leading coefficient of $f$ and $\Delta$ is the discriminant of the largest monic squarefree divisor of $f$.

\item If $f$ is of type $(2,2,0,0,0)$, then
\begin{align*}
\#\{P\in E(\kbar(t))\mid x(P)\in k[t], \hat{h}(P)=1/3\}&=2\cdot \#\{\text{odd theta characteristics in }\Pic(C_2)\},\\
\#\left\{P\in E(\kbar(t))\mid y(P)\in k[t], \hat{h}(P)=1/3\right\}&=3\cdot\#\{P_1+P_3,P_1+P_4,P_2+P_3,P_2+P_4\}(k),\\
\#\{P\in E(\kbar(t))\mid x(P)\in k[t], \hat{h}(P)=2/3\}&=4(\#\{\text{even theta characteristics in }\Pic(C_2)\}-1),\\
\#\left\{P\in E(k(t))\mid \hat{h}(P)=2/3\right\}&= \#\{P_1,P_2,P_3,P_4\}(k)\cdot\\
&\quad\#\{\text{even theta characteristics in }\Pic(C_2)\}-1),\\
\#\{P\in E(\kbar(t))\mid x(P)\in k[t], \hat{h}(P)=1\}&=2\cdot \#\{\text{odd theta characteristics in }\Pic(C_2)\},\\
\#\left\{P\in E(\kbar(t))\mid y(P)\in k[t], \hat{h}(P)=1\right\}&=3\cdot\#\{P_1'+P_3',P_1'+P_4',P_2'+P_3',P_2'+P_4'\}(k),
\end{align*}
where
\begin{align*}
\{P_1,P_2,P_3,P_4\}= 
\begin{cases}
t^{-1}(\{f_2=0\}),\quad &\text{if $\deg f_2=2$}\\
t^{-1}(\{f_2=0\})\cup t^{-1}(\infty), \quad &\text{if $\deg f_2=1$},
\end{cases}
\end{align*}
$t(P_1)=t(P_2), t(P_3)=t(P_4)$, and $P_1',P_2',P_3',P_4'$ are analogously defined on $C_3'$.

\item If $f$ is of type $(4,1,0,0,0)$, then 
\begin{align*}
\#\{P\in E(\kbar(t))\mid x(P)\in k[t], \hat{h}(P)=2/3\}&=2(\#\{\text{odd theta characteristics in }\Pic(C_2)\}-1),\\
\#\{P\in E(\kbar(t))\mid y(P)\in k[t], \hat{h}(P)=2/3\}&=3(\#\Pic(C_3,\mathbb{Q}.D_3)[3]-\#\Pic(C_3)[3]),\\
\#\{P\in E(\kbar(t))\mid x(P)\in k[t], \hat{h}(P)=1\}&=2\cdot \#\{\text{even theta characteristics in }\Pic(C_2)\}.
\end{align*}

\item If $f$ is of type $(6,0,0,0,0)$, then 
\begin{align*}
\#\{P\in E(\kbar(t))\mid x(P)\in k[t], \hat{h}(P)=1\}&=2\cdot \#\{\text{odd theta characteristics in }\Pic(C_2)\},\\
\#\{P\in E(\kbar(t))\mid y(P)\in k[t],\hat{h}(P)=1\}&=3(\#\Pic(C_3)[3]-1).
\end{align*}
\end{enumerate}
\end{thm}

We note that Theorem~\ref{corSh2} refines and extends certain results of Bremner \cite{Bremner} ($k=\mathbb{Q}, \deg f\leq 3$), Shioda \cite[Th.~8.2]{Shioda} ($k=\kbar$, $f$ squarefree), and Elkies \cite{Elkies} ($\deg f=5,6$, $f$ squarefree), the last of whom already pointed out the connection with $2$ and $3$-descent maps, and with theta characteristics. See also work of Clebsch \cite{Clebsch1869} and van Geemen \cite{vGeemen1992}. It seems likely that the present methods can also be used to construct similar expressions for the rank of $E(k(t))$ in each of the cases of Theorem \ref{corSh2} (this rank was computed by Bremner \cite{Bremner} when $k=\mathbb{Q}$ and $\deg f\leq 3$; see also \cite{ALRM07} for an alternative approach using Nagao's conjecture to compute ranks in certain cases).

In order to illustrate these connections, let us clarify the earlier work of Bremner \cite{Bremner}, which in particular studied the set of integral points $E(\Q[t])$ when $f(t)$ is a squarefree cubic (Type $(3,0,1,0,0)$). In this case, let us detail how one can recover from Theorem~\ref{corSh2} the classification of nonexceptional integral points given by Bremner \cite[Theorem~1.2]{Bremner}. We first use Theorem~\ref{corSh2} to characterize when there is an integral point $P=(x(t),y(t))\in E(\Q[t])$ of height $\hat{h}(P)=\frac{1}{2}$ (equivalently, $\deg x=1, \deg y\leq 1$). If such a point exists, then clearly the leading coefficient $a$ of $f$ must be a perfect cube in $\Q$ (consistent with Theorem~\ref{corSh2}~\ref{1.2vi}, if this is not the case, $\Pic(C_2,\mathbb{Q}.D_2)[2]\setminus\Pic(C_2)[2]$ is empty). Suppose now that $a$ is a perfect cube. If $P\in E(\bar\Q[t]), \hat{h}(P)=\frac{1}{2}$, and $y\in \Q[t]$, then it follows easily from the fact $a$ is a perfect cube that after multiplying $x$ by an appropriate $3$rd root of unity,  $x\in \Q[t]$ as well. Then by Theorem~\ref{corSh2}, and since $\Q$ does not contain a primitive third root of unity, assuming $a$ is a perfect cube, we have
\begin{align*}
\#\{P\in E(\Q[t])\mid \hat{h}(P)=1/2\}=\frac{1}{3}\#\{P\in E(\bar\Q(t))\mid y(P)\in \Q[t], \hat{h}(P)=1/2\}=\#\Pic(C_3)[3]-1.
\end{align*}
The curve $C_3:y^2=f(t)$ is an elliptic curve, and it is well-known that $C_3$ has either $0$ or $2$ nontrivial $\mathbb{Q}$-rational $3$-torsion points. Thus, we see that 
\begin{align}
\label{Brem1}
\#\{P\in E(\Q[t])\mid \hat{h}(P)=1/2\}=
\begin{cases}
2 \quad &\text{ if }a\in\mathbb{Q}^{*3} \text{ and }C_3[3]\neq \{0\},\\
0 &\text{ otherwise}.
\end{cases}
\end{align}

For integral points of canonical height $1$, by Theorem~\ref{corSh2} we have a bijection
\begin{align*}
\{P\in E(\Q[t])\mid \hat{h}(P)=1\}\leftrightarrow (\Pic(C_2)[2]\setminus \{0\})\times(\Pic(C_3)[3]\setminus\{0\}).
\end{align*}

An easy calculation shows that $C_2:x^3=-f(t)$ is an elliptic curve which can be given by the Weierstrass equation $Y^2=X^3+16\Delta_f$, where $\Delta_f$ is the discriminant of $f$. Then $C_2$ has a single nontrivial $\mathbb{Q}$-rational $2$-torsion point if and only if $2\Delta_f$ is a perfect cube in $\mathbb{Q}$. Thus, we find
\begin{align}
\label{Brem2}
\#\{P\in E(\Q[t])\mid \hat{h}(P)=1\}=
\begin{cases}
2 \quad &\text{ if }2\Delta_f\in\mathbb{Q}^{*3} \text{ and }C_3[3]\neq \{0\},\\
0 &\text{ otherwise}.
\end{cases}
\end{align}

Using the above simple characterizations, one may parametrize the various possibilities, leading precisely to the (seemingly complicated) description of Bremner in \cite[Theorem~1.2]{Bremner} (excluding information on the exceptional integral points, which were also classified by Bremner). We carry out the explicit calculations in Remark~\ref{rmk:Bremner}.

When $k$ is algebraically closed, we obtain explicit formulas for the number of nonexceptional integral points using standard facts about the cardinalities of torsion subgroups of Picard groups, the genus formulas \eqref{eq:genusofC2} and \eqref{eq:genusofC3}, and the well-known formula for the number of odd (resp.~even) theta characteristics on a curve of genus $g$: $2^{g-1}(2^{g}-1)$ (resp.~$2^{g-1}(2^{g}+1)$).  This extends Theorem~8.2 of \cite{Shioda}, where these quantities were computed when $f$ is squarefree, $\deg f\leq 6$, and additionally all nonexceptional integral points were found when $\deg f\leq 4$ (``extra integral points" in \cite{Shioda}). Remarkably, when $k=\kbar$, our method allows one to precisely count sets of integral points in two distinct ways: as twice the number of elements in an associated $2$-torsion subset of a Picard group (expressed in powers of $2$) and  as three times the number of elements in an associated $3$-torsion subset of a Picard group (expressed in powers of $3$).  The most interesting cases are given in Table~\ref{table1}, where we have also included entries on the associated Mordell-Weil lattices. The last entry in Table~\ref{table1} (type $(6,0,0,0,0)$) is discussed extensively by Elkies \cite{Elkies}, who makes many further connections, including the connection with lattices discussed below, and the history and prior work of Clebsch and van Geemen mentioned previously.

\begin{table}
\caption{Some entries from Theorem~\ref{corSh2} over $\kbar$}
\label{table1}
\centering
\begin{tabular}{|c|c|c|c|c|c|}
\hline
Type & Mordell-Weil & $\hat{h}$ & $\#E(\kbar[t])_h$ or& $2\cdot\#$ Picard & $3\cdot$\# Picard\\
 & Lattice $L$& & $\#\{\mathbf{v}\in L\mid |\mathbf{v}|^2=2h\}$ & $2$-torsion subset & $3$-torsion subset\\
\hline
$(2,0,0,1,0)$ & $A_2^*$ & $\frac{1}{3}$ & $6$ & $2^3-2$ & $3^2-3$\\
\hline
$(3,0,1,0,0)$ & $D_4^*$ & $\frac{1}{2}$ & $24$ & $2^5-2^3$ & $3^3-3$\\
\hline
$(4,1,0,0,0)$ & $E_6^*$ & $\frac{2}{3}$ & $54$ & $2^6-2^3-2$ & $3^4-3^3$\\
 & 	 & $1$ & $72$ & $2^6+2^3$ & $3^4-3^2$\\
\hline
$(6,0,0,0,0)$ & $E_8$ & $1$ & $240$ & $2^8-2^4$ & $3^5-3$\\
\hline
\end{tabular}
\end{table}

The table yields the following identities:
\begin{align}
6&=2^3-2=3^2-3\label{unit1}\\
24&=2^5-2^3=3^3-3\label{unit2}\\
240&=2^8-2^4=3^5-3.\label{unit3}
\end{align}
and
\begin{align}
54&=2^6-2^3-2=3^4-3^3\label{unit4}\\
72&=2^6+2^3=3^4-3^2.\label{unit5}
\end{align}

The identities \eqref{unit1}--\eqref{unit3} were noted by Pillai in \cite{Pillai}, where it was conjectured that these were the only solutions to the exponential Diophantine equation
\begin{align*}
2^x-2^y=3^z-3^w
\end{align*}
in positive integers $x,y,z,w$ with $x>y$ and $z>w$.  Using estimates from linear forms in logarithms, Pillai's conjecture was proven by Stroeker and Tijdeman in \cite{ST}.   We obtain then an ``explanation" of Pillai's three identities, where we note that the numbers $6$, $24$, and $240$ arise both as the number of integral points of minimal height in $E(\kbar[t])$, where $E:y^2=x^3+f(t)$, $f\in \kbar[t]$ is squarefree and $\deg f=2,3, 5$ (or $6$), respectively, and as the number of vectors of minimal norm (in other words, the kissing number) of the lattices $A_2^*, D_4^*$, and $E_8^*=E_8$, respectively.  In fact, these numbers are precisely the kissing numbers in dimensions $2$, $4$, and $8$, respectively (i.e., the maximal possible kissing number in the respective dimensions; this is known, in these cases, for both lattice packings and arbitrary sphere packings).

The identity \eqref{unit5} gives a nontrivial solution in positive integers to the equation 
\begin{align*}
2^x+2^y=3^z-3^w,
\end{align*}
whose finitely many solutions were classified by Pillai \cite{Pillai}.  Finally, the identity \eqref{unit4} gives a solution in positive integers to the equation
\begin{align}
2^x-2^y-2^z=3^v-3^w>0. \label{unit6}
\end{align}
If we exclude the solutions which come from the identities \eqref{unit1}--\eqref{unit3} (i.e., when $y=z$ or $x\in \{y+1,z+1\}$), then \eqref{unit4} and the identity
\begin{align}
216&=2^8-2^5-2^4=3^5-3^3\label{unit7}
\end{align}
appear to yield the only solutions in positive integers to \eqref{unit6} (up to permuting $y$ and $z$).   The solutions to \eqref{unit6} can, in principle, be effectively computed using a (unpublished) method of Bennett (see \cite[p.~650]{levin14}).


\section{Extending Picard groups}
\label{sec:EPic}


Let $C$ be a smooth projective curve over a field $k$, and let $D=\sum_{i=1}^r P_i$ be a reduced divisor on $C$, where the $P_i$ are closed points. We define the group of divisors with rational coefficients above $D$ by letting
$$
\Div(C,\Q.D):=\Div(C\setminus D)\oplus\bigoplus_{i=1}^r \Q\cdot P_i,
$$
and we define $\Pic(C,\Q.D)$ by
$$
\Pic(C,\Q.D):=\Div(C,\Q.D)/\PDiv(C),
$$
where $\PDiv(C)$ is the usual group of principal divisors. It was shown in \cite{gillibert09} that this group is isomorphic to the group of $\Gm$-torsors for Kato's log flat topology on the curve $C$ endowed with the log structure attached to $D$. We refer to \cite[Section~3]{gillibert09} for further details, and proofs of the statements that we recall below. In fact, using the definition above, these statements have elementary proofs.

We have a short exact sequence
\begin{equation}
\label{eq:ratPic1}
\begin{CD}
0 @>>> \Pic(C) @>>> \Pic(C,\Q.D) @>\nu>> \bigoplus_{i=1}^r (\Q/\Z)\cdot P_i @>>> 0. \\
\end{CD}
\end{equation}

The usual degree map on Cartier divisors induces a degree map $\deg:\Pic(C,\Q.D)\to \Q$, which fits inside a commutative diagram
\begin{equation}
\label{eq:ratdeg}
\begin{CD}
\Pic(C,\Q.D) @>\nu>> \bigoplus_{i=1}^r (\Q/\Z)\cdot P_i \\
@V\deg VV @VV\textrm{sum} V \\
\Q @>>> \Q/\Z \\
\end{CD}
\end{equation}
where the map $\nu$ is that from \eqref{eq:ratPic1}, the horizontal bottom map is the natural one, and the right vertical map is the sum map $\sum_{i=1}^r a_i\cdot P_i\mapsto \sum_{i=1}^r a_i$.

Given $n>1$, the Snake Lemma allows us to deduce from \eqref{eq:ratPic1} an exact sequence
\begin{equation}
\label{eq:ratPictorsion}
\begin{CD}
0 \to \Pic(C)[n] \to \Pic(C,\Q.D)[n] @>\nu_n>> \bigoplus_{i=1}^r (\frac{1}{n}\Z/\Z)\cdot P_i \to \Pic(C)/n \to \Pic(C\setminus D)/n.
\end{CD}
\end{equation}

Torsion elements of $\Pic(C,\Q.D)$ have degree zero, and so it follows from \eqref{eq:ratdeg} that the image of the map $\nu_n$ is contained in the kernel of the sum map. Therefore, when $n=p$ is a prime, we are able to deduce from \eqref{eq:ratPictorsion} an upper bound
\begin{equation}
\label{eq:ratPicbound}
\dim_{\F_p} \Pic(C,\Q.D)[p] \leq \dim_{\F_p} \Pic(C)[p] + r -1,
\end{equation}
which is valid unless $r=0$, in which case the divisor $D$ is empty and $\Pic(C,\Q.D)=\Pic(C)$.

On the other hand, we have a Kummer-like exact sequence
\begin{equation}
\label{eq:kummerlog}
\begin{CD}
0 @>>> k^{\times}/n @>>> H^1(C\setminus D,\mu_n) @>\kappa>> \Pic(C,\Q.D)[n] @>>> 0, \\
\end{CD}
\end{equation}
which we shall now describe in more detail. The group $H^1(C\setminus D,\mu_n)$ can be identified as a subgroup of $H^1(k(C),\mu_n)=k(C)^\times/n$ as follows:
$$
H^1(C\setminus D,\mu_n)=\{h\in k(C)^\times/n ~|~ \forall v\in C\setminus D,\quad v(h)\equiv 0\pmod{n}\}.
$$
Then the map $\kappa$ in \eqref{eq:kummerlog} is simply given by
\begin{equation}
\label{eq:kappadef}
\begin{split}
\kappa:H^1(C\setminus D,\mu_n) &\longrightarrow \Pic(C,\Q.D)[n] \\
h\in k(C)^\times &\longmapsto \frac{1}{n}\divisor(h),
\end{split}
\end{equation}
and the exactness of \eqref{eq:kummerlog} can be proved by using this description.


\section{Heights}
\label{sec:heights}


Throughout this paper, we will assume the following hypotheses:

\begin{enumerate}\renewcommand{\labelenumi}{\theenumi}\renewcommand{\theenumi}{(Hyp~\arabic{enumi})}
  \item $k$ is a perfect field of characteristic not $2$ or $3$. \label{field_k}
    \item $E$ is the elliptic curve over $k(t)$ defined by the Weierstrass equation \label{curve_kt}
\begin{equation}
\tag{$E$}
y^2=x^3+f(t).
\end{equation}
  \item $f=f_1f_2^2f_3^3f_4^4f_5^5$ where the $f_i$ are pairwise coprime, separable polynomials in $k[t]$. We let $d_i:=\deg f_i$ and $d:=\deg f=d_1+2d_2+3d_3+4d_4+5d_5$. \label{f12}
  \item $f_1$ is nonconstant. \label{f1nonct}
\end{enumerate}


The height of a point on an elliptic curve measures, in some sense, the complexity of the point.  If $E$ is an elliptic curve over $K=k(t)$, given in some Weierstrass form, and $P=(x,y)\in E(K)$, then we can define the \emph{naive height} by
\begin{equation*}
h(P)=\max \left\{\frac{1}{2}\deg x,\frac{1}{3}\deg y\right\}.
\end{equation*}
Here, if $z\in k(t)$, $z=\frac{p}{q}$, $p,q\in k[t]$, $(p,q)=1$, then $\deg z=\max\{\deg p,\deg q\}$ (i.e., the degree of the corresponding rational function on $\mathbb{P}^1$).

A more fundamental notion of height on $E$ is given by the \emph{canonical height} $\hat{h}$.  This agrees with the naive height $h$ up to $O(1)$ and gives a quadratic form on $E(K)$, providing a link to the group structure on $E$.  While the naive height depends on the chosen Weierstrass equation for $E$, the canonical height is intrinsic to the curve $E$.  Following Tate, the canonical height may be simply defined as $\hat{h}(P)=\lim_{n\to \infty} 4^{-n}h(2^nP)$, where $h$ is some naive height for $E$.

We now explicitly compute the canonical height on the elliptic curves $E$ of interest here.
\begin{thm}
\label{thh}
Assume hypotheses~\ref{field_k} -- \ref{f12} above, and let $h$ (resp. $\hat{h}$) be the naive (resp. canonical) height on $E$. Given  $P=(x,y)\in E(k(t))$ we let, for $1\leq i\leq 5$,
\begin{equation*}
n_i:=\#\{\alpha\in \kbar\mid f_i(\alpha)=0, x(\alpha)=0 \}.
\end{equation*}
Then
\begin{equation*}
\hat{h}(P)=h(P)-\frac{1}{3}n_2-\frac{1}{2}n_3-\frac{2}{3}n_4.
\end{equation*}
\end{thm}
Since it is easily seen that $n_1=n_5=0$, this may be rewritten more compactly as $\hat{h}(P)=h(P)-\frac{1}{6}\sum_{i=1}^5in_i$.
\begin{cor}
Under the same hypotheses as Theorem~\ref{thh}, if $f$ is square-free then $\hat{h}=h$, i.e., the canonical and naive heights agree.
\end{cor}

\begin{proof}[Proof of Theorem~\ref{thh}]
Let $\Emin\to \PP^1_k$ be the minimal regular model of $E$, whose zero section we denote by $O$. Abusing notation, we denote by $P$ the section of $\Emin$ corresponding to the point $P$. Then, according to \cite[Theorem~8.6]{Shioda2}, the canonical height of $P$ is given by the formula
$$
\hat{h}(P) = \frac{1}{2} \langle P,P\rangle = \chi(\Emin) + (P).(O) - \frac{1}{2}\sum_{\textrm{bad}~v} \contr_v(P)
$$
where $\langle P,P\rangle$ is the N{\'e}ron-Tate height pairing on $\Emin$, $\chi(\Emin)$ is the arithmetic genus of $\Emin$, and $\contr_v(P)$ is a local contribution of $P$ at a place $v$ of bad reduction. This local contribution is zero if and only if $P$ meets the identity component of the fiber of $\Emin$ above $v$ (equivalently, $P$ has good reduction at $v$).

According to Lemma~\ref{lem:badE}, the bad places of $\Emin$ are those dividing $f$ and possibly $\infty$. The following table gives the reduction type of $\Emin$ at a place $v$ dividing $f$, and the local contribution of a point $P$ which has bad reduction at $v$, extracted from \cite[table~(8.16)]{Shioda2}
\begin{center}
\begin{tabular}{c|cccccc}
 $v$ divides & $f_1$ & $f_2$ & $f_3$ & $f_4$ & $f_5$ \\
 \hline
 reduction type at $v$ & $\mathrm{II}$ & $\mathrm{IV}$ & $\mathrm{I}_0^*$ & $\mathrm{IV}^*$ & $\mathrm{II}^*$ \\
 $\contr_v(P)$ for bad $P$  & 0 & 2/3 & 1 & 4/3 & 0 \\
\end{tabular}
\end{center}
It follows immediately that
$$
\frac{1}{2}\sum_{\textrm{bad}~v\neq\infty} \contr_v(P) = \frac{1}{3}n_2 + \frac{1}{2}n_3 + \frac{2}{3}n_4.
$$
So, in order to prove the statement, it remains to prove that the naive height satisfies
\begin{equation}
\label{eq:height_agree}
h(P) = \chi(\Emin) + (P).(O) - \frac{1}{2}\contr_\infty(P).
\end{equation}
Firstly, we have
$$
\chi(\Emin) = \left\lceil\frac{\deg f}{6}\right\rceil.
$$
Since $k[t]$ is a principal ideal domain, one can write
$$
P = (x,y) = \left(\frac{a}{e^2},\frac{b}{e^3}\right) \qquad \text{with~} a,b,e \in k[t] \text{~and~} \gcd(e,ab)=1.
$$
Since $f$ is sixth-power free, the Weierstrass equation $y^2=x^2+f$ is minimal at all places, except possibly at $\infty$. It follows that $P$ meets $O$ above some place $v\neq \infty$ if and only if $v$ divides $e$, in which case the intersection multiplicity is $\ord_v(e)$. So, the intersection number $(P).(O)$ is obtained by adding $\deg e$ with the intersection multiplicity above infinity.
In order to compute the latter, we consider the Weierstrass equation  $E':y^2=x^3+t^{6\chi}f(\frac{1}{t})$ and the point $P'=(t^{2\chi}x(\frac{1}{t}),t^{3\chi}y(\frac{1}{t}))$ on $E'$ corresponding to $P$. It now suffices to compute the intersection $(P').(O)$ at $t=0$.

One can write
$$
P'=\left(t^{2\chi}\frac{a(\frac{1}{t})}{e(\frac{1}{t})^2},t^{3\chi}\frac{b(\frac{1}{t})}{e(\frac{1}{t})^3})\right) = \left(\frac{U}{t^{\deg a - 2\deg e -2\chi}},\frac{V}{t^{\deg b - 3\deg e -3\chi}}\right)
$$
where $U$ and $V$ are $t$-adic units. It follows that the denominator $e'$ of $P'$ has $t$-valuation
$$
\max\left\{\left\lceil\frac{1}{2}\deg a\right\rceil - \deg e -\chi, \left\lceil\frac{1}{3}\deg b\right\rceil - \deg e -\chi, 0\right\}.
$$
Summing up everything, we obtain
\begin{equation}
\label{eq:max}
\chi + (P).(O) = \max\left\{\left\lceil\frac{1}{2}\deg a\right\rceil, \left\lceil\frac{1}{3}\deg b\right\rceil, \deg e + \chi \right\}.
\end{equation}
We deduce from the relation $b^2=a^3+fe^6$ that among the three polynomials $b^2$, $a^3$, $fe^6$, at least two of them have the same degree, hence two of the three following quantities agree:
\begin{equation}
\label{eq:maxbis}
\frac{1}{2}\deg a;\quad \frac{1}{3}\deg b;\quad \deg e + \frac{\deg f}{6}.
\end{equation}
If the maximum in \eqref{eq:maxbis} is achieved simultaneously by the first two quantities, then $P$ has good reduction at $\infty$, and $\frac{1}{2}\deg a$ and $\frac{1}{3}\deg b$ are integers. So, we obtain from \eqref{eq:max} that
$$
\chi + (P).(O)  - \frac{1}{2}\contr_\infty(P) = \chi + (P).(O) = \max\left\{\frac{1}{2}\deg a, \frac{1}{3}\deg b\right\} = h(P),
$$
which proves \eqref{eq:height_agree}.
Now, if the maximum in \eqref{eq:maxbis} is achieved by the last quantity, it means that $P$ has bad reduction at $\infty$, and that either $\frac{1}{2}\deg a=\deg e + \frac{\deg f}{6}$ or $\frac{1}{3}\deg b = \deg e + \frac{\deg f}{6}$. In each case, one can check from Lemma~\ref{lem:badE} and \cite[table~(8.16)]{Shioda2} that
$$
\frac{1}{2}\contr_\infty(P) = \left\lceil\frac{\deg f}{6}\right\rceil- \frac{\deg f}{6}  = \chi-\frac{\deg f}{6},
$$
which, combined with \eqref{eq:max}, yields \eqref{eq:height_agree}.
\end{proof}

The following Lemma describes the bad fibers of N{\'e}ron's minimal regular model $\Emin\to\PP^1$ of the elliptic curve defined by the equation $(E)$ above.

\begin{lem}
\label{lem:badE}
Over $\overline{k}$, the places of bad reduction of $\Emin\to \PP^1$ are the zeroes of $f$, and possibly the point at infinity. The reduction types are as follows:
\begin{center}
\begin{tabular}{c|cccccc}
 $v$ divides & $f_1$ & $f_2$ & $f_3$ & $f_4$ & $f_5$ \\
 \hline
 reduction type at $v$ & $\mathrm{II}$ & $\mathrm{IV}$ & $\mathrm{I}_0^*$ & $\mathrm{IV}^*$ & $\mathrm{II}^*$ \\
\end{tabular}
\end{center}
At infinity, the reduction type depends on the congruence class of $d$ modulo $6$ as follows:
\begin{center}
\begin{tabular}{c|cccccc}
 $d \pmod{6}$ & $0$ & $1$ & $2$ & $3$ & $4$ & $5$ \\
 \hline
 reduction type at $\infty$ & good & $\mathrm{II}^*$ & $\mathrm{IV}^*$ & $\mathrm{I}_0^*$ & $\mathrm{IV}$ & $\mathrm{II}$ \\
\end{tabular}
\end{center}
\end{lem}

\begin{proof}
The first table follows from Tate's algorithm \cite{tate75}. It remains to see what happens at infinity. For that, we should look at the  behavior at $t=0$ of the elliptic curve defined by the equation $y^2=x^3+f(\frac{1}{t})$. This equation can be rewritten as $y^2=x^3+t^{6\chi}f(\frac{1}{t})$ where $\chi=\left\lceil\frac{\deg f}{6}\right\rceil$ as previously. Using Tate's algorithm \cite{tate75}, one deduces the reduction type (over $\overline{k}$) at $t=0$, depending on the value of $d \pmod{6}$.
\end{proof}


\section{The $2$-descent map}


\begin{prop}
\label{prop:2descent}
Assume hypotheses~\ref{field_k} -- \ref{f12} above, and that $f$ is not a cube in $\kbar[t]$.
Let $C_2$ be the smooth, projective, geometrically integral $k$-curve defined by the affine equation:
\begin{equation}
\tag{$C_2$}
x^3=-f(t).
\end{equation}
Let us define a divisor $D_2\subset C_2$ by letting
$$
D_2:= t^{-1}(\{f_3=0\})\cup \left\{
\begin{array}{ll}
t^{-1}(\infty) & \text{if}~ d \equiv 3\pmod{6}\\
\emptyset & \text{otherwise}.
\end{array}\right.
$$

Then the map
\begin{align*}
\phi_2:E(k(t))/2E(k(t)) &\longrightarrow \Pic(C_2,\Q.D_2)[2] \\
(x_0,y_0) &\longmapsto \frac{1}{2}\divisor(x_0-x)
\end{align*}
is an injective morphism of groups.
\end{prop}

\begin{rmq}
It follows from the Riemann-Hurwitz formula that the genus of $C_2$  is given by
\begin{equation}
\label{eq:genusofC2}
g(C_2)= \left\{
\begin{array}{ll}
d_1+d_2+d_4+d_5-2 & \text{if}~ 3\mid d \\
d_1+d_2+d_4+d_5-1 & \text{otherwise}.
\end{array}\right.
\end{equation}
\end{rmq}

By combining the Proposition above with the bound \eqref{eq:ratPicbound}, one obtains the following Corollary, which also follows easily from \cite[Theorem~1.1]{gl18a}.

\begin{cor}
\label{cor:2descent}
In the set-up of Proposition~\ref{prop:2descent}, we have 
\begin{equation}
\label{eq:rankbound1.1}
\rk_\Z E(k(t)) \leq \dim_{\F_2} \Pic(C_2)[2] + \sum_{v\in\PP^1} \varepsilon_v
\end{equation}
where
$$
\varepsilon_v:=\left\{
\begin{array}{ll}
2 & \text{if the fiber of $D_2$ at $v$ is the sum of three $k_v$-rational points} \\
1 & \text{if the fiber of $D_2$ at $v$ has exactly one $k_v$-rational point} \\
0 & \text{otherwise}.
\end{array}\right.
$$
It follows that
\begin{equation}
\label{eq:rankbound1.2}
\rk_\Z E(k(t)) \leq \left\{
\begin{array}{ll}
2\cdot\left(\sum d_i -2\right) & \text{if}\enspace 6\mid d \\
2\cdot\left(\sum d_i -1\right) & \text{otherwise}.
\end{array}\right.
\end{equation}
Moreover, when $\zeta_3\notin k$, this upper bound can be improved by a factor two.
\end{cor}

\begin{rmq}
While the bound \eqref{eq:rankbound1.2} is none other than the geometric rank bound deduced from Igusa's inequality, the bound \eqref{eq:rankbound1.1} has a more arithmetic flavor. See the introduction of \cite{gl18a} for further details.
\end{rmq}

\begin{proof}[Proof of Proposition~\ref{prop:2descent}]
It is well known (see for example \cite[\S{}2]{BK1977}) that the composition of the cohomological maps defined in \cite[\S{}2.1]{gl18a}
$$
\begin{CD}
E(k(t))/2 @>\delta>> H^1(k(t),E[2]) @>h^1\omega>> H^1(k(C_2),\mu_2), \\
\end{CD}
$$
can be described as
\begin{align*}
E(k(t))/2 &\longrightarrow k(C_2)^\times/2\simeq H^1(k(C_2),\mu_2) \\
(x_0,y_0) &\longmapsto x_0-x,
\end{align*}
where $x$ is the $x$-coordinate map on $C_2$, given by the equation $x^3=-f(t)$. Let $\Egp\to \PP^1$ be N{\'e}ron's group scheme model of $E$, which is the smooth locus of $\Emin\to \PP^1$, and let $\Phi_v$ be the group of connected components of the special fiber of $\Egp$ at a given place $v$. The arguments at the beginning of the proof of \cite[Theorem~1.1]{gl18a} prove that the Kummer map $\delta$ has values in the subgroup $H^1(\PP^1\setminus \Sigma_2,\Egp[2])$, where $\Sigma_2$ is the set of places of bad reduction of $\Egp$ at which $\#\Phi_v(\overline{k})$ is divisible by $2$. It follows that the map $h^1\omega\circ \delta$ has values in $H^1(C_2\setminus t^{-1}(\Sigma_2),\mu_2)$.

According to Lemma~\ref{lem:badE}, the bad fibers at places dividing $f_3$ contribute to $\Sigma_2$, and the contribution at infinity occurs exactly when $d\equiv 3 \pmod{6}$, in which case $\Phi_\infty(\overline{k})\simeq (\Z/2)^2$. Therefore,
$$
t^{-1}(\Sigma_2)=D_2
$$
and the map $h^1\omega\circ \delta$ has values in $H^1(C_2\setminus D_2,\mu_2)$.

If follows from the discussion above that $\phi_2$ is none other than the composition of $h^1\omega\circ \delta$ with the map $\kappa$ defined in \eqref{eq:kappadef}. Therefore, $\phi_2$ is a morphism of groups. The injectivity of $\phi_2$ follows from the same argument as in the proof of \cite[Theorem~1.1]{gl18a}, the key ingredients being: the injectivity of $h^1\omega\circ \delta$, the exactness of the sequence \eqref{eq:kummerlog}, and the fact that $k^\times/2$ does not contribute to the kernel of the norm map.
\end{proof}

\begin{proof}[Proof of Corollary~\ref{cor:2descent}]
Let us prove \eqref{eq:rankbound1.1}. According to Proposition~\ref{prop:2descent}, $\phi_2$ is injective, with values in $\Pic(C_2,\Q.D_2)[2]$.
Let us compute the number of closed points of $D_2$. We note that $D_2=t^{-1}(\Sigma_2)$, where $\Sigma_2$ is as in the proof of Proposition~\ref{prop:2descent}.

The map $t:D_2\to \Sigma_2$ is {\'e}tale of degree $3$, from which it follows that the number of closed points of $D_2$ lying above a given $v\in\Sigma_2$ is given by the formula
$$
\#(D_2)_v=\left\{
\begin{array}{ll}
3 & \text{if}~ \# (D_2)_v(k_v)=3 \\
2 & \text{if}~ \# (D_2)_v(k_v)=1 \\
1 & \text{otherwise}.
\end{array}\right.
$$

Moreover, it follows from a local argument (in short: one can define the $2$-descent map over the localization of $\PP^1$ at each $v\in\Sigma_2$) that the image of the map $\phi_2$ lands in the kernel of the multi-degree map
$$
\Pic(C_2,\Q.D_2)[2] \longrightarrow \bigoplus_v (\Q/\Z)\cdot v,
$$
where each partial map computes the degree (mod $\Z$) of the fiber of the divisor above $v$.

Observing that $\varepsilon_v=\#(D_2)_v-1$, it follows from a variant of \eqref{eq:ratPicbound} that
$$
\rk_\Z E(k(t)) \leq \dim_{\F_2} \Pic(C_2)[2] + \sum_v \varepsilon_v.
$$

It remains to prove \eqref{eq:rankbound1.2}. If $6\mid d$, then $(D_2)_\infty=\emptyset$ and $\varepsilon_\infty=0$. On the other hand, the other $\varepsilon_v$ are bounded by $2$, and their number is bounded above by $d_3=\deg f_3$. Moreover, $g(C_2)=\sum_{i\neq 3} d_i-2$ according to \eqref{eq:genusofC2}. It follows from \eqref{eq:rankbound1.1} that
$$
\rk_\Z E(k(t)) \leq \dim_{\F_2} \Pic(C_2)[2] + \sum_v \varepsilon_v \leq 2\cdot (\sum_{i\neq 3} d_i-2) + 2d_3.
$$

If $d\equiv 3\pmod{6}$, then $\varepsilon_\infty$ is bounded by $2$, and $g(C_2)=\sum_{i\neq 3} d_i-2$, hence
$$
\rk_\Z E(k(t)) \leq 2\cdot (\sum_{i\neq 3} d_i-2) + 2d_3 +2.
$$

In any other case, $\varepsilon_\infty=0$ and $g(C_2)=\sum_{i\neq 3} d_i-1$, hence
$$
\rk_\Z E(k(t)) \leq 2\cdot (\sum_{i\neq 3} d_i-1) +2d_3.
$$
which proves the result.

Finally, if $\zeta_3\notin k$, the field $k(\zeta_3)$ is a quadratic extension of $k$, and we have
$$
\rk_\Z E(k(\zeta_3,t)) = \rk_\Z E(k(t)) + \rk_\Z E'(k(t)),
$$
where $E'$ denotes the twist of $E$ by the extension $k(\zeta_3)/k$. According to the discussion at the beginning of the next section, $E'$ is isogenous to $E$ over $k(t)$, hence $E$ and $E'$ have the same rank over $k(t)$. This implies that the rank of $E$ over $k(\zeta_3,t)$ is twice the rank over $k(t)$. The bound \eqref{eq:rankbound1.2} being valid over $\overline{k}(t)$, one deduces the result.
\end{proof}

\begin{rmq}
Corollary~\ref{cor:2descent} can be recovered from the third statement of \cite[Theorem~1.1]{gl18a}.
\end{rmq}


\section{The $3$-descent maps}


Before we proceed with $3$-descent, let us recall some well-known facts. Let $E'$ be the elliptic curve defined by the equation
\begin{equation}
\tag{$E'$}
y^2=x^3-27f(t).
\end{equation}
Then $E'$ is the twist of $E$ by the quadratic extension $k(\zeta_3)=k(\sqrt{-3})$, unless $\zeta_3$ belongs to $k$, in which case $E'$ is isomorphic to $E$. One checks that the map
\begin{align*}
\lambda:E &\longrightarrow E' \\
(x,y) &\longmapsto \left(\frac{x^3+4f}{x^2},\frac{y(x^3-8f)}{x^3}\right)
\end{align*}
is a $3$-isogeny, whose kernel is the subset of $E$ defined by the equation $x=0$. Consequently, the smooth compactification of $\ker(\lambda:\Egp\to\Egp')\setminus\{0\}$ is the projective curve $C_3$ defined by the affine equation $y^2=f(t)$ (where $\Egp$ and $\Egp'$ denote N{\'e}ron's group scheme models). This curve $C_3$ is geometrically integral, because $f$ is not a square in $\overline{k}(t)$ according to \ref{f1nonct}.

We denote by $\lambda^t:E'\to E$ the dual isogeny. By combining $\lambda$-descent and $\lambda^t$-descent, we shall prove the following.

\begin{prop}
\label{prop:3descent}
Assume hypotheses~\ref{field_k} -- \ref{f12}, and that $f$ is not a square in $\kbar[t]$.
Let $C_3$ and $C_3'$ be the smooth, projective, geometrically integral $k$-curves defined by the affine equations below:
\begin{align}
y^2&=f(t)\tag{$C_3$} \\
y^2&=-27f(t).\tag{$C_3'$}
\end{align}
We define a divisor $D_3\subset C_3$ by letting
$$
D_3:= t^{-1}(\{f_2f_4=0\})\cup \left\{
\begin{array}{ll}
t^{-1}(\infty) & \text{if}\enspace d\equiv 2 ~\text{or}~ 4\pmod{6}\\
\emptyset & \text{otherwise}
\end{array}\right.
$$
and the same formula defines a similar divisor $D_3'\subset C_3'$. Then:
\begin{enumerate}
\item[1)] We have a short exact sequence
$$
0 \longrightarrow E'(k(t))/\lambda E(k(t)) \longrightarrow E(k(t))/3E(k(t)) \longrightarrow  E(k(t))/\lambda^t E'(k(t)) \longrightarrow 0\\
$$
in which the left-hand side map is induced by $\lambda^t$.
\item[2)] The maps
\begin{align*}
\phi_3:E(k(t))/\lambda^t E'(k(t)) &\longrightarrow \Pic(C_3,\Q.D_3)[3] \\
(x_0,y_0) &\longmapsto \frac{1}{3}\divisor(y_0-y)
\end{align*}
and
\begin{align*}
\phi_3':E'(k(t))/\lambda E(k(t)) &\longrightarrow \Pic(C_3',\Q.D_3')[3] \\
(x_0,y_0) &\longmapsto \frac{1}{3}\divisor(y_0-y)
\end{align*}
are injective morphisms of groups.
\end{enumerate}
\end{prop}

The kernels of $\lambda$ and $\lambda^t$ are Cartier dual to each other. In particular, $C_3$ and $C_3'$ become isomorphic over the field $k(\zeta_3)=k(\sqrt{-3})$, which is clear when looking at their equations.

\begin{rmq}
We observe that $C_3$ can be more simply defined by the equation $Y^2=f_1(t)f_3(t)f_5(t)$, and similarly for $C_3'$.
It follows (again) from the Riemann-Hurwitz formula that the genus of $C_3$ and $C_3'$ are given by
\begin{equation}
\label{eq:genusofC3}
g(C_3)=g(C_3')=\left\{
\begin{array}{ll}
\frac{1}{2}(d_1+d_3+d_5-2) & \text{if}\enspace 2\mid d\\\frac{1}{2}(d_1+d_3+d_5-1) & \text{otherwise}.
\end{array}\right.
\end{equation}
\end{rmq}

\begin{cor}
\label{cor:3descent}
In the set-up of Proposition~\ref{prop:2descent}, we have
\begin{equation}
\label{eq:rankbound2.1}
\rk_\Z E(k(t)) \leq \dim_{\F_3} \Pic(C_3)[3] + \dim_{\F_3} \Pic(C_3')[3] +\varepsilon +\varepsilon'
\end{equation}
where
$$
\varepsilon:=\#\{v\in\PP^1, \text{the fiber of $D_3$ at $v$ is the sum of two $k_v$-rational points}\}
$$
and $\varepsilon'$ is defined similarly for $D_3'$.
\end{cor}

\begin{rmq}
When $k$ is algebraically closed, the bound \eqref{eq:rankbound2.1}, like the bound \eqref{eq:rankbound1.1}, is equivalent to the geometric rank bound \eqref{eq:rankbound1.2} deduced from Igusa's inequality.
Nevertheless, \eqref{eq:rankbound1.1} and \eqref{eq:rankbound2.1} are not equivalent in general.
\end{rmq}

\begin{proof}[Proof of Proposition~\ref{prop:3descent}]
1) This follows from the ker-coker Lemma for the map $[3]=\lambda^t\circ\lambda$ on the group $E(k(t))$, the exactness on the left being due to the fact that $\ker(\lambda^t)$ has no rational point over $k(t)$. 2) The properties of $\phi_3$ and $\phi_3'$ can be proved with the same methods as in the $2$-descent case.
\end{proof}

\begin{proof}[Proof of Corollary~\ref{cor:3descent}]
By combining the two statements of Proposition~\ref{prop:3descent}, and observing that $E(k(t))$ has no $3$-torsion, one finds that
$$
\rk_\Z E(k(t)) \leq \dim_{\F_3} \Pic(C_3,\Q.D_3)[3]+\dim_{\F_3} \Pic(C_3',\Q.D_3')[3].
$$
The inequality \eqref{eq:rankbound2.1} follows by similar arguments as in the proof of Corollary~\ref{cor:2descent}.
\end{proof}



\section{Upper bounds for the number of integral points}
\label{sec:ub}


\subsection{Davenport's inequality}


In addition to the tools developed in the previous sections, we will also use the following inequality of Davenport \cite{Dav}, with an extension by Sch{\"u}tt and Schweizer \cite{SS} to positive characteristic.

\begin{thm}[Davenport's inequality]
\label{davenport}
Let $k$ be a field and let $g,h\in k[t]$ with $\deg(g^3)=\deg(h^2)=6M$ for some integer $M>0$.  Suppose that one of the following holds:
\begin{enumerate}
\item[(i)] $k$ has characteristic $0$ and $g^3\neq h^2$;
\item[(ii)] $k$ has characteristic not $2$ or $3$, and $g^3-h^2$ contains at least one irreducible factor of multiplicity one.
\end{enumerate}
Then
$$
\deg (g^3-h^2)\geq M+1.
$$
\end{thm}

\begin{proof}
The case (i) is the original Davenport's inequality. The case (ii) is a result by Sch{\"u}tt and Schweizer \cite[Theorem~1.2, (b)]{SS}.
\end{proof}

Recall from section~\ref{sec:heights} that for a point $P=(x(t),y(t))\in E(k(t))$, we define the naive height of $P$ by
\begin{equation*}
h(P):=\max \left\{\frac{1}{2}\deg x,\frac{1}{3}\deg y\right\}.
\end{equation*}

Davenport's inequality can be restated in terms of heights as follows.

\begin{cor}
\label{Dcor}
Assume hypotheses~\ref{field_k} -- \ref{f1nonct} above. If $P\in E(k[t])$ is an integral point on the elliptic curve $E$, then $h(P)\leq d-1=\deg f -1$.
\end{cor}

It will be convenient to set
\begin{align*}
E(k[t])_{n}&=\{P\in E(k[t])\mid h(P)= n\}.
\end{align*}
We similarly define sets $E(k[t])_{<n}$, $E(k[t])_{\geq n}$, etc.


\subsection{Tools from $2$-descent}
\label{subsect:2descent}


The aim is to prove the following:

\begin{thm}\label{thm:main}
Assume hypotheses~\ref{field_k} -- \ref{f1nonct} above. Assume in addition that $f=f_1f_2^2f_4^4$ and that $3\nmid d$. For $i=2,4$, let $\omega_i=\omega(f_i)$ be the number of irreducible factors of $f_i$. Let us consider the natural map
\begin{align*}
E(k[t])\to E(k(t))/2E(k(t)).
\end{align*}
\begin{enumerate}
\item\label{item:mainthsmallht} When restricted to~$E(k[t])_{\leq d/6}$, this map is $2$-to-$1$ onto its image and omits $0$. In particular,
$$
\#E(k[t])_{\leq d/6} \leq 2^{\rk_{\Z} E(k(t))+1}-2 \leq 2^{2d_1+2d_2-1}-2.
$$
\item\label{item:mainthlargeht} When restricted to~$E(k[t])_{> d/6}$, this map is at most $2^{\omega_2+\omega_4+2}$-to-$1$ onto its image. In particular,
$$
\# E(k[t])_{>d/6} \leq 2^{\rk_{\Z} E(k(t))+\omega_2+\omega_4+2} \leq 2^{2d_1+3d_2+d_4}.
$$
\item\label{item:mainth} The set~$E(k[t])$ of all integral points on~$E$ satisfies
\begin{align*}
\#E(k[t])&\leq 2^{\rk_{\Z} E(k(t))+1}(2^{\omega_2+\omega_4+1}+1)-2\leq 2^{2d_1+2d_2-1}\left(2^{\omega_2+\omega_4+1}+1\right)-2.
\end{align*}
\item\label{item:mainthfeqf1} If in addition $f=f_1$, the previous bound can be sharpened as follows:
\begin{align*}
\#E(k[t])&\leq 2^{\rk_{\Z} E(k(t))+2}-2\leq 2^{2d}-2.
\end{align*}
\end{enumerate}
\end{thm}

\begin{rmq}
\label{d2rem}
When $d=2$, one can slightly improve the inequality of Theorem~\ref{thm:main} (iv) to obtain a sharp result.  In this case, the height of any point in $E(k(t))$ is at least $\frac{1}{3}$.  Then Davenport's inequality, combined with the fact that $h$ is a canonical height, shows that the map of Theorem~\ref{thm:main} omits $0$ in its image.  In this case, we obtain the improved bound $\#E(k[t])\leq 2^{2d-1}-2=6$, which is sharp by \cite[Th.~8.2]{Shioda} when $k$ is algebraically closed of characteristic $0$.
\end{rmq}

As before, let $C_2$ be the smooth projective completion of the affine plane curve defined by $x^3+f(t)=0$.  We first record some elementary facts about the geometry of $C_2$.

If $3\nmid d$, the unique pole of~$t$ in~$k(t)$ totally ramifies in~$k(C_2)$ and we denote by~$\infty$ the corresponding
point on~$C_2$. If $3\mid d$, the pole of~$t$ totally splits in~$k(C_2)$ and we denote by~$\infty_1,\infty_2,\infty_3$ the corresponding points
of~$C_2$ (over $\kbar$).

We can identify the function field $k(C_2)$ with $k(t,x)$.  For $h\in k(C_2)$, we let $(h)_\infty$ denote the divisor of poles of $h$.
When $h\in k[t]$, we use $\deg_t h$ to denote the degree of $h$ as a polynomial in $t$.  Note that, when considered as maps $C_2\to \PP^1$, $x$ has degree $d$, and $t$ has degree $3$.

For a divisor $D$ on $C_2$, we let
\begin{equation}
  \ord^\star_\infty(D)
=
\begin{cases}
  \text{multiplicity of~$D$ at~$\infty$} &\text{if~$3\nmid d$}\\
  \max_{1\leq i\leq 3}\left\{\text{multiplicity of~$D$ at~$\infty_i$}\right\} &\text{if~$3\mid d$}\\
\end{cases}
\end{equation}

\begin{lem}\label{RRlem}
Assume that $f=f_1f_2^2f_4^4$. Then the ring of rational functions on $C_2$ which are regular outside infinity can be explicitly described as
$$
k[t] \oplus k[t]\cdot \frac{x}{f_4} \oplus k[t]\cdot \frac{x^2}{f_2f_4^2}.
$$
Moreover, if~$\psi=\gamma +\beta\frac{x}{f_4} +\alpha\frac{x^2}{f_2f_4^2}$
belongs to this ring, then the divisor of its poles satisfies
\begin{align*}
\ord_\infty^\star (\psi)_\infty
=
\begin{cases}
\max\{3\deg_t \gamma, 3\deg_t \beta + d_1+2d_2+d_4, 3\deg_t \alpha +2d_1+d_2+2d_4 \}
&\text{if $3\nmid d$},\\
\max\left\{\deg_t \gamma, \deg_t \beta + \frac{d_1+2d_2+d_4}{3}, \deg_t \alpha +\frac{2d_1+d_2+2d_4}{3}\right\}
&\text{if $3\mid d$,}
\end{cases}
\end{align*}
(note that the rationals appearing in this formula are all integers).
\end{lem}

\begin{proof}
The ring~$B$ of rational functions on $C_2$ which are regular outside infinity is nothing else than the integral closure of~$k[t]$
inside~$k(C_2) = k(t,x)$. Since the extension~$k(t,x)/k(t)$ is radical of degree~$3$ and since~$3$ does not divide the
characteristic of~$k$, only the zeros of~$f_1f_2f_4$ ramify in~$k(t,x)$. More precisely, every zero of~$f_1f_2f_4$ totally ramifies
 in~$k(t,x)$ and~$\dis(B/k[t]) = f_1^2f_2^2f_4^2$. The two functions~$\frac{x}{f_4}$ and~$\frac{x^2}{f_2f_4^2}$ are elements of~$B$
since~$\bigl(\frac{x}{f_4}\bigr)^3 + f_1f_2^2f_4=0$ and~$\bigl(\frac{x^2}{f_2f_4^2}\bigr)^3 - f_1^2f_2f_4^2= 0$.
We deduce that~$k[t]\oplus k[t]\frac{x}{f_4}\oplus k[t]\frac{x^2}{f_2f_4^2} \subset B$. Using {\tt magma}, we compute the discriminant
of the left ring; we find~$f_1^2f_2^2f_4^2$ and therefore both rings are equal.

For~$\psi\in B$, the computation of~$\ord_\infty^\star (\psi)_\infty$ is strongly related to Newton polygons.
Let~$v = -\deg_t$ be the valuation ``at infinity'' of the rational function field~$k(t)$. This valuation totally ramifies (resp.~splits) in~$k(C_2)$
if~$3\nmid d$ (resp.~$3\mid d$); let~$w$ or~$w_1,w_2,w_3$ denote the extensions of~$v$
to~$k(C_2)$ (note that~$w(B) = \frac{1}{3}\Z$). The number~$\ord_\infty^\star (\psi)_\infty$ is related with these valuations, in the following way:
$$
\ord_\infty^\star (\psi)_\infty
=
\begin{cases}
  -3w(\psi) &\text{if~$3\nmid d$}\\
  \max_{1\leq i\leq 3}\left\{-w_i(\psi)\right\} &\text{if~$3\mid d$}
\end{cases}
$$
Hence, up to a factor of~$3$ in the case~$3\nmid d$, we want to compute the maximum at~$\psi$ of the opposite valuations extending~$v$.
These opposite valuations are known to correspond to the slopes of the Newton polygon of the characteristic polynomial of~$\psi$, and the maximum of
these opposite valuations must be the slope of the right-most segment of the polygon (because the polygon is a lower convex hull its slopes must
increase from left to right). 
Let~$\chi(X)=X^3+c_2X^2+c_1X+c_0\in k[t][X]$ be the characteristic polynomial of~$\psi$; with {\tt magma}, we easily find that
$$
\chi(X) = X^3
- 3 \gamma X^2
+ 3\left(\alpha \beta f_1 f_2 f_4 + \gamma^2\right) X
- \left(\alpha^3 f_1^2 f_2 f_4^2 - \beta^3 f_1 f_2^2 f_4 + 3 \alpha \beta \gamma f_1 f_2 f_4 + \gamma^3\right).
$$
In order to study the Newton polygon of this polynomial, we put
\begin{align*}
\delta_0 &= v(\gamma),
&
\delta_1 &= v(\beta) - \frac{d_1+2d_2+d_4}{3},
&
\delta_2 &= v(\alpha) - \frac{2d_1+d_2+2d_4}{3},
&
\delta &= \max\left\{-\delta_0, -\delta_1, -\delta_2\right\}.
\end{align*}
We tabulate the $v$-values of the different products which appear in the coefficients of~$\chi(X)$:
$$
\begin{array}{c||c||c|c||c|c|c|c}
&c_2X^2 & \multicolumn{2}{c||}{c_1X} & \multicolumn{4}{c}{c_0}\\
\hline
\text{coef.}
&
\gamma
&
\alpha \beta f_1 f_2 f_4
&
\gamma^2
&
\alpha^3 f_1^2 f_2 f_4^2
&
\beta^3 f_1 f_2^2 f_4
&
\alpha \beta \gamma f_1 f_2 f_4
&
\gamma^3\\
\hline
v
&
\delta_0
&
\delta_1+\delta_2
&
2\delta_0
&
3\delta_2
&
3\delta_1
&
\delta_0+\delta_1+\delta_2
&
3\delta_0
\end{array}
$$
Hence we have~$v(c_i) \geq \delta(i-3)$.
This means that the points of the Newton polygon of~$\chi(X)$ are above the line~$j=\delta(i-3)$, i.e. at the gray points or above in the
following figure.
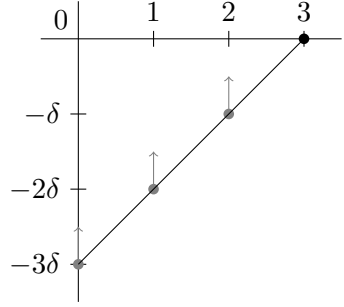
\begin{center}
\begin{tikzpicture}
\draw (-0.5,0) -- (3.5,0);
\draw (0,0.5) -- (0,-3.5);
\draw (0,0) node[above left] {0};
\foreach \x in {1,2,3} \draw (\x,3pt) node[above] {\x} -- (\x,-3pt) ;
\draw (3pt,-1) -- (-3pt,-1) node[left,anchor=east] {$-\delta$}; 
\foreach \y in {2,3} \draw (3pt,-\y) -- (-3pt,-\y) node[left,anchor=east] {$-\y\delta$}; 
\fill (3,0) circle (2pt);
\fill[gray] (2,-1) circle (2pt);\draw[gray,->] (2,-1) -- (2,-0.5);
\fill[gray] (1,-2) circle (2pt);\draw[gray,->] (1,-2) -- (1,-1.5);
\fill[gray] (0,-3) circle (2pt);\draw[gray,->] (0,-3) -- (0,-2.5);
\draw (0,-3) -- (3,0);
\end{tikzpicture}
\end{center}
In fact, at least one of the gray points are on the segment; indeed:
\begin{itemize}
\item if~$\delta = -\delta_0$, the right gray point~$(2,-\deg_t(c_2))$ lies on the segment;
\item if~$\delta$ is equal to exactly one of the~$(-\delta_i)$'s then the left gray point~$(0,-\deg_t(c_0))$ lies on the segment;
\item in the remaining case, when~$\delta = -\delta_1=-\delta_2<-\delta_0$, then the middle gray point~$(1,-\deg_t(c_1))$ lies on the segment.
\end{itemize}
Therefore the right-most segment of the polygon must be of slope~$\delta$. This proves that~$\delta = -w(\psi)$ if~$3\nmid d$
or~$\delta = \max_{1\leq i\leq 3} \left\{-w_i(\psi)\right\}$ if~$3\mid d$ and ends the proof.
\end{proof}

\begin{lem}[Key Lemma]
\label{lem:start}
Assume that $f=f_1f_2^2f_4^4$. Consider the composition
$$
E(k[t])\to E(k(t))/2E(k(t))\hookrightarrow\Pic(C_2,\mathbb{Q}.D_2)[2].
$$
If two points $P_0=(x_0(t),y_0(t))$ and $P_1=(x_1(t),y_1(t))$ in $E(k[t])$ have the same image by this map, then there exists a rational function
$\psi\in k(C_2)$, regular outside infinity, such that
\begin{equation}
\label{eq:psidef}
(x_0(t)-x)(x_1(t)-x)=\psi^2.
\end{equation}
Moreover:
\begin{enumerate}
\item there exist polynomials $a$, $b$ and $c$ in $k[t]$, with~$c\neq 0$, such that
$$
\psi =c+bx+a\frac{x^2}{f_2f_4^2}
$$
and we have the relations
\begin{subequations}
\begin{align}
& b^2+2\frac{ac}{f_2f_4^2} =1 \label{c2} \\
& 2bc-a^2f_1 = -(x_0+x_1) \label{c1} \\
& c^2-2abf_1f_2f_4^2 =x_0x_1. \label{c0}
\end{align}
\end{subequations}
\item\label{item:ifaeq0} if $a=0$ then $P_0=\pm P_1$.
\end{enumerate}
\end{lem}

\begin{rmq}
In Lemma~\ref{lem:start} (i), the statement that $c\neq 0$ fails when $f_1$ is constant. But in fact there are counterexamples to Theorem~\ref{thm:main} in this case, as was proved by Sch{\"u}tt and Schweizer \cite[Theorem~1.2 (a)]{SS}.
\end{rmq}

\begin{proof}
The above composition associates to $(x_0(t),y_0(t))$ the divisor class of the divisor $D_0=\frac{1}{2}\divisor(x_0(t)-x)$ in $\Pic(C_2,\mathbb{Q}.D_2)[2]$, where we view $x_0(t)-x$ as a rational function in $k(C_2)$.  Obviously, the two distinct points $(x_0(t),y_0(t))$ and $(x_0(t),-y_0(t))$ map to the same element of $\Pic(C_2,\mathbb{Q}.D_2)[2]$ (or $E(k(t))/2E(k(t))$).  Let $(x_1(t),y_1(t))\in  E(k[t])$ be another point and let $D_1=\frac{1}{2}\divisor(x_1(t)-x)$.  Suppose that $[D_0]=[D_1]$ in $\Pic(C_2,\mathbb{Q}.D_2)[2]$, or equivalently, $D_0-D_1$ is a principal divisor.  Let $\divisor(\phi)=D_0-D_1$, where $\phi\in k(C_2)$.  Then $\divisor(\phi^2)=2D_0-2D_1=\divisor\left(\frac{x_0(t)-x}{x_1(t)-x}\right)$ and it follows that
$$
(x_0(t)-x)(x_1(t)-x)=\psi^2
$$
for some $\psi\in k(C_2)$.  Since the function on the left-hand side is regular outside infinity, one deduces that $\psi$ is regular outside infinity.

According to Lemma~\ref{RRlem}, there exist polynomials $\alpha$, $\beta$ and $\gamma$ such that
$$
\psi =\gamma+\beta \frac{x}{f_4} +\alpha\frac{x^2}{f_2f_4^2}.
$$

By identifying the coefficients of $1$, $x$ and $x^2$ in the identity \eqref{eq:psidef}, we find that
\begin{subequations}
\begin{align}
& \frac{\beta^2}{f_4^2}+2\frac{\alpha\gamma}{f_2f_4^2} =1 \label{b1} \\
& 2\frac{\beta\gamma}{f_4}-\alpha^2f_1 = -(x_0+x_1) \label{b2} \\
& \gamma^2-2\alpha\beta f_1f_2f_4 =x_0x_1. \label{b3}
\end{align}
\end{subequations}

First, we shall prove that $f_4$ divides $\beta$. Assume by contradiction that this is not the case; then there exists an irreducible factor $\pi$ of $f_4$ which does not divide $\beta$. According to \eqref{b2}, $\frac{\beta\gamma}{f_4}$ is a polynomial, hence $\pi\mid \gamma$. On the other hand, according
to \eqref{b1}, we have
$$
\beta^2= f_4^2-2\frac{\alpha\gamma}{f_2}.
$$
But we know that $\pi\nmid f_2$ because $f_2$ and $f_4$ are coprime. It follows that $\pi$ divides the polynomial $\frac{\alpha\gamma}{f_2}$, and hence $\pi$ divides $\beta^2$, a contradiction.

This proves that $f_4$ divides $\beta$. In order to prove (i), it remains to prove that $\gamma\neq 0$. Assume by contradiction that $\gamma=0$; then it follows from \eqref{b3} that $f_1$ divides $x_0x_1$. But $f_1$ is nonconstant \ref{f1nonct} hence, up to exchanging $x_0$ and $x_1$, the polynomials $f_1$ and $x_1$ have an irreducible factor in common, that we denote by $\pi$.
On the other hand, the point $P_1=(x_1,y_1)$ lies on $E$, hence
$$
y_1^2=x_1^3+f_1f_2^2f_4^4.
$$
The polynomial $f_1$ being separable, its $\pi$-adic valuation is exactly one. Therefore, the $\pi$-adic valuation of $x_1^3+f_1f_2^2f_4^4$ is also one, which is impossible since this polynomial is a square.

It remains to prove (ii). If $\alpha=0$, then, up to a choice of sign for $\psi$, we find that
$$
\beta=f_4, \quad 2\gamma =-(x_0+x_1), \quad \gamma^2=x_0x_1,
$$
which implies that $x_0=x_1=-\gamma$. It follows that $P_0=\pm P_1$.

Finally, letting $a:=\alpha$, $b:=\beta/f_4$ and $c:=\gamma$, the relations \eqref{b1}, \eqref{b2} and \eqref{b3} are none other than the relations at the end of the statement.
\end{proof}

\begin{proof}[Main steps of the proof of Theorem~\ref{thm:main}]
The goal of the proof is to estimate the cardinality of fibers of
the natural map~$E(k[t]) \to E(k(t))/2E(k(t))$. To this end we use the key Lemma~\ref{lem:start}: let $P_0=(x_0(t),y_0(t))$ and $P_1=(x_1(t),y_1(t))$
be two points in $E(k[t])$ such that $P_0$ and $P_1$ define the same class in $E(k(t))/2E(k(t))$; then there exists a rational
function~$\psi\in k(C_2)$, regular outside infinity, such that
\begin{align}\label{eq:start}
&\psi^2=(x_0(t)-x)(x_1(t)-x)
&
&\text{and}
&
&\psi = c + b x + a\frac{x^2}{f_2f_4^2}
\end{align}
for some~$a,b,c\in k[t]$ with~$c\not=0$ and satisfying~\eqref{c0},\eqref{c1},\eqref{c2}.
Thanks to Lemma~\ref{RRlem}, these two expressions give two ways of computing~$\ord_\infty^\star (\psi)_\infty$.
At that stage, it is natural to distinguish between the cases of points of \emph{small height} ($\leq \frac{d}{6}$) and of \emph{large height}
($> \frac{d}{6}$) since the divisor of poles of the function~$x_i(t)-x$  behaves differently in each case. In any case, thanks to Lemma~\ref{lem:start}~\ref{item:ifaeq0},
we are reduced to proving that~$a = 0$; see the proofs of items~\ref{item:mainthsmallht}, \ref{item:mainthlargeht} and~\ref{item:mainthfeqf1} below for the details.

Once the fibers are bounded above in cardinality by a power of~$2$, we easily deduce an upper bound for~$\# E(k[t])$ that involves the size
of the target set inside~$E(k(t))/2E(k(t))$, which is at most~$2^{\rk_\Z E(k(t))}$ (or this size minus~$1$). Finally, from the ``geometric'' upper bound
on $\rk_\Z E(k(t))$ stated in Corollary~\ref{cor:2descent}, 
we are able to deduce an upper bound on $\# E(k[t])$ which depends only on the $d_i$ and the $\omega_i$.
\end{proof}

\begin{proof}[Proof of Theorem~\ref{thm:main}~\ref{item:mainthsmallht}]
The two formulas in~$\eqref{eq:start}$ give two ways to compute~$\ord_\infty^\star (\psi)_\infty$.
On one hand,
by Lemma~\ref{RRlem}, we know that~$\ord_\infty^\star\left(x_i(t)-x\right)_\infty = \max\{3\deg_t x_i, d\} = d$, 
which yields
\begin{equation}
\ord_\infty^\star(\psi)_\infty = d.
\end{equation}
On the other hand, by the same Lemma,
\begin{equation}
\ord_\infty^\star(\psi)_\infty = \max\left\{3\deg_t c, 3\deg_t b + d, 3\deg_t a + 2d_1 + d_2 + 2d_4\right\}.
\end{equation}
Comparing these two formulas and taking into account the facts that~$d = d_1+2d_2+4d_4$ and that~$3\nmid d$ leads to
\begin{subequations}
  \begin{align}
    3\deg_t c&< d \label{small_ht_deg_0}\\
    \deg_t b &\leq 0 \label{small_ht_deg_1}\\
    3\deg_t a + d_1 &\leq d_2 + 2d_4.\label{small_ht_deg_2}
  \end{align}
\end{subequations}
The second inequality means that~$b$ is a constant polynomial;
the combination~$\frac{1}{3}\left[\eqref{small_ht_deg_0}+\eqref{small_ht_deg_2}\right]$ yields~$\deg_t a + \deg_t c < d_2+2d_4$.
Since by \eqref{c2}, $2ac = (1-b^2)f_2f_4^2$, we must have~$ac = 0$, and thus~$a=0$ because~$c\not=0$.
\end{proof}

\begin{proof}[Proof of Theorem~\ref{thm:main}~\ref{item:mainthlargeht}]
First, if an integral point~$P=(x(t),y(t))$ satisfies~$h(P) > \frac{d}{6}$, then, by definition of the height, either~$3\deg_t x(t) > d$ or~$2\deg_t y(t) > d$. Since~$y(t)^2 = x(t)^3 + f(t)$, with~$\deg_t f(t) =d$, necessarily, both degrees are greater than~$d$
and~$3\deg_t x(t) = 2\deg_t y(t)$. Consequently,~$2\mid \deg_t x(t)$, $3\mid \deg_t y(t)$
and~$h(P) = \frac{\deg_t x(t)}{2} = \frac{\deg_t y(t)}{3}$ is a positive integer.

Secondly, according to Davenport's inequality (Corollary~\ref{Dcor}), we have
$$
E(k[t])_{>d/6} = \{P\in E(k[t]), d/6 < h(P) \leq d-1\}.
$$
We now write $E(k[t])_{>d/6}$ as the union of two subsets:
$$
E(k[t])_{>d/6} = \{P\in E(k[t]), d/6 < h(P) \leq d/2 \} \cup \{P\in E(k[t]), d/2 < h(P) \leq d-1 \}.
$$
We observe that, if $P$ belongs to one of these subsets, then $3h(P)$ is an upper bound on the height of elements in the same subset. The result then follows by applying Lemma~\ref{tinj2} below to each of these two subsets.
\end{proof}

\begin{lem}
\label{tinj2}
Let $P_i=(x_i(t),y_i(t))\in E(k[t])_{>d/6}$, $i=0,1$.  If $h(P_1)\leq h(P_0)< 3h(P_1)$, $\gcd(x_0(t),f_2f_4)=\gcd(x_1(t),f_2f_4)$, and $P_0$ and $P_1$ have the same image in $E(k(t))/2E(k(t))$, then $P_0=\pm P_1$.
\end{lem}

\begin{proof}
As in the case of points of small height, the formulas contained in~\eqref{eq:start} give two ways to compute~$\ord_\infty^\star (\psi)_\infty$. 
On one hand, by Lemma~\ref{lem:start}, since~$\deg_t x_i > \frac{d}{3}$, we know
that
$$
\ord_\infty^\star\left(x_i(t)-x\right)_\infty = \max\left\{3\deg_t x_i, d\right\} = 3\deg_t x_i.
$$
Since~$\psi^2=(x_0(t)-x)(x_1(t)-x)$, we have
\begin{equation}\label{eq:poles_x0x1}
\ord_\infty^\star(\psi)_\infty = 3\left(\frac{\deg_t x_0}{2}+\frac{\deg_t x_1}{2}\right) = 3\left(h(P_0)+h(P_1)\right).
\end{equation}
On the other hand, by Lemma~\ref{lem:start}, we also know that
\begin{equation}\label{eq:poles_abc}
\ord_\infty^\star(\psi)_\infty = \max\left\{3\deg_t c, 3\deg_t b + d_1+2d_2+d_4, 3\deg_t a + 2d_1+d_2+2d_4\right\}.
\end{equation}
We have that~$d_1+2d_2+d_4 \equiv d \pmod{3}$,~$2d_1+d_2+2d_4 \equiv 2d \pmod{3}$ and~$3 \nmid d$, hence~$3\deg_t c$ is the only integer that is divisible
by~$3$ in the preceding maximum. Putting together \eqref{eq:poles_x0x1}
and~\eqref{eq:poles_abc}, we deduce that
\begin{subequations}
\begin{align}
  \deg_t c = h(P_0)+h(P_1) \label{delta0}\\
  3\deg_t b + d_1+2d_2+4d_4 < 3\deg_t c \label{delta1}\\
  3\deg_t a + 2d_1+d_2+2d_4 < 3\deg_t c. \label{delta2}
\end{align}
\end{subequations}

We want to prove that~$a=0$, and so we assume, by contradiction, that~$a\not=0$. Recall that $c\not= 0$ by Lemma~\ref{lem:start}. Then by
formula~\eqref{c2} one has~$b^2-1 = \frac{2ac}{f_2f_4^2}$.

  If~$b = 0$, we deduce that~$\deg_t(a) = d_2 + 2d_4 - \deg_t c < d_2 + 2d_4$ since~$\deg_t c = h(P_0)+h(P_1) > 0$. This inequality still holds
if~$b\not=0$. Indeed, first observe that the two nonzero polynomials~$(b^2-1)f_2f_4^2$ and~$2ac$ have the same degree, that is
\begin{equation}\label{eq:d2plus2d4}
d_2+2d_4 = \deg_t a + \deg_t c - 2\deg_t b.
\end{equation}
Replacing this expression of~$d_2+2d_4$ in the sum of~\eqref{delta1} and~\eqref{delta2}
leads to~$2\deg_t a + d_1 < \deg_t b + \deg_t c$. Using~\eqref{c1}, we deduce that~$\deg_t b + \deg_t c = \deg(x_0+x_1)$
and then
\begin{align*}
  \deg_t b &\leq \max\{2h(P_0),2h(P_1)\} - \deg_t c &&\text{since~$\deg_t x_i = 2h_i$}\\
           &\leq 2h(P_0) - \deg_t c && \text{since~$h(P_1)\leq h(P_0)$ by hypothesis}\\
           &\leq 2h(P_0) - (h(P_0)+h(P_1)) = h(P_0)-h(P_1)&&\text{by~\eqref{delta0}}\\
           &< \frac{h(P_0)+h(P_1)}{2}  &&\text{since~$h(P_0) < 3h(P_1)$ by hypothesis.}
\end{align*}
Hence~$\deg_t c - 2\deg_t b > 0$ and \eqref{eq:d2plus2d4} implies that~$\deg_t a < d_2+2d_4$.

According to Lemma~\ref{lem:irreducible_pi}, there exists an irreducible~$\pi\in k[t]$ that divides~$x_0,x_1,f_2f_4,c$ but not~$a$.
Thanks to~\eqref{c1},~$\pi \mid a^2f_1$; this is impossible since~$\gcd(f_1,f_2f_4)=1$ and~$\pi\nmid a$. Therefore~$a=0$.
\end{proof}

\begin{lem}\label{lem:irreducible_pi}
  Suppose that~$\deg_t a < d_2 + 2d_4$ and that~$\gcd(x_0,f_2f_4) = \gcd(x_1,f_2f_4)$. Then there exists an irreducible polynomial that
divides~$x_0,x_1,f_2f_4$ and~$c$ but not~$a$.
\end{lem}

\begin{proof}
By \eqref{c2}, we know that~$f_2f_4^2$ divides~$ac$. Assume, by contradiction, that $f_2f_4$ divides $a$. Let $a':=a/f_2$; then since
$\deg_t a < d_2 + 2d_4$ we have that $\deg a' < 2d_4$, and since $f_4^2$ divides $a'c$ there must exist an irreducible polynomial~$\pi\in k[t]$ which divides~$f_4,a$ and~$c$.

By~\eqref{c0}, necessarily~$\pi$ divides~$x_0x_1$, and hence it divides one of the~$x_i$ and thus
it divides both of them since~$\gcd(x_0,f_2f_4) = \gcd(x_1,f_2f_4)$. Going back to the Weierstrass equation~$y_i^2 = x_i^3 + f_1f_2^2f_4^4$,
we deduce that, in fact,~$\pi^2$ must divide~$x_0$ and~$x_1$. Then~$\pi^3 \mid abf_1f_2f_4^2$, $\pi^4 \mid x_0x_1$ and by~\eqref{c0},
necessarily~$\pi^2 \mid c$. Using~\eqref{c0} again, we prove that~$\pi^2 \mid ab$. We remark that we cannot have~$\pi \mid b$
since otherwise, by~\eqref{c2}, this would imply that~$\pi \mid 1$. Hence~$\pi^2 \mid a$. Since this is true for every irreducible polynomial
dividing~$f_4$, we deduce that~$f_2f_4^2$ divides~$a$, which contradicts the assumption that~$\deg_t a < d_2 + 2d_4$.
\end{proof}

\begin{proof}[Proof of Theorem~\ref{thm:main}~\ref{item:mainthfeqf1}.] 
When the polynomial~$f$ is separable, i.e. when~$f=f_1$, the upper bound can be sharpened. Instead of dividing the height
range~$\left[0,d\right[$ in three pieces~$\left[0,\frac{d}{6}\right]$, $\left]\frac{d}{6},\frac{d}{2}\right]$
and~$\left]\frac{d}{2},d\right[$, we divide it in two pieces only, namely~$\left[0,\frac{d}{3}\right[$ and~$\left[\frac{d}{3},d\right[$.

Firstly, we prove that the natural map~$E(k[t])_{<d/3} \to E(k(t))/2E(k(t))$ is $2$-to-$1$ onto its image and omits~$0$. As usual, the two
formulas~$\eqref{eq:start}$ give two ways of computing~$\ord_\infty^\star (\psi)_\infty$.
On one hand, since~$3\deg_t x_i < 2d$, by Lemma~\ref{RRlem}, we know that~$\ord_\infty^\star\left(x_i(t)-x\right)_\infty = \max\{3\deg_t x_i, d\} < 2d$
which yields~$\ord_\infty^\star(\psi)_\infty < 2d$. On the other hand, by the same
Lemma, if~$a\not=0$, then~$\ord_\infty^\star(\psi)_\infty = \max\left\{3\deg_t c, 3\deg_t b + d, 3\deg_t a + 2d\right\} \geq 2d$, and we reach
a contradiction. Therefore~$a=0$ and Lemma~\ref{lem:start} allows us to conclude the argument.

Secondly, we prove that the natural map~$E(k[t])_{\geq d/3} \to E(k(t))/2E(k(t))$ is $2$-to-$1$ onto its image: if two points~$P_0,P_1 \in E(k[t])_{\geq d/3}$
with~$h(P_1)\leq h(P_0)$ have the same image then~$h(P_0) < 3h(P_1)$ and Lemma~\ref{tinj2} allows us to conclude the argument.

The sharpened upper bound follows.
\end{proof}

It turns out that a mix of the techniques used to bound the number of integral points of small and large height when~$3\nmid d$ can
be gathered to obtain an upper bound on the number of integral points of small height when~$3\mid d$.

\begin{prop}\label{prop:smallheigt3midd}
Assume hypotheses~\ref{field_k} -- \ref{f1nonct} above. Assume in addition that $f=f_1f_2^2f_4^4$ and that $3\mid d$. For $i=2,4$, let $\omega_i=\omega(f_i)$ be the number of irreducible factors of $f_i$. Let us consider the natural map
\begin{align*}
E(k[t])\to E(k(t))/2E(k(t)).
\end{align*}
When restricted to~$E(k[t])_{\leq d/6}$, this map is at most $2^{\omega_2+\omega_4+1}$-to-$1$ onto its image, and omits $0$. In particular,
$$
\#E(k[t])_{\leq d/6} \leq 2^{\rk_{\Z} E(k(t))+\omega_2+\omega_4+1}-2
\leq \left\{
\begin{array}{ll}
2^{2d_1+3d_2+d_4-3} -2 & \text{if}\enspace d\equiv 0\pmod{6} \\
2^{2d_1+3d_2+d_4-1} -2 & \text{if}\enspace d\equiv 3\pmod{6}.
\end{array}\right.
$$
\end{prop}

\begin{proof}
  We begin as in the proof of Theorem~\ref{thm:main}. By considering the two expressions of the function~$\psi\in k(C_2)$ in~\eqref{eq:start}, 
Lemma~\ref{RRlem} yields
\begin{equation}
\ord_\infty^\star(\psi)_\infty
=
\frac{d}{3}
=
\max\left\{
\deg_t c, \deg_t b + \frac{d}{3}, \deg_t a + \frac{2d_1+d_2+2d_4}{3}
\right\}.
\end{equation}
Since~$d = d_1+2d_2+4d_4$, we deduce that
\begin{subequations}
  \begin{align}
    3\deg_t c&\leq d \label{c0bis}\\
    \deg_t b &\leq 0 \label{c1bis}\\
    3\deg_t a + d_1 &\leq d_2 + 2d_4. \label{c2bis}
  \end{align}
\end{subequations}

Unlike the case when~$3\nmid d$, the first inequality need not be strict. Nevertheless, either this inequality is
strict, in which case we can conclude as in the proof of Theorem~\ref{thm:main},~\ref{item:mainthsmallht}, or it
is an equality, i.e.~$3\deg_t c = d = d_1 + 2d_2 + 4d_4$. Then, using this to eliminate~$d_1$ in~\eqref{c2bis}
yields~$\deg_t a \leq d_2+2d_4 - \deg_t c < d_2+2d_4$ since~$\deg_t c = d/3 > 0$.

According to Lemma~\ref{lem:irreducible_pi}, there exists an irreducible~$\pi\in k[t]$ which divides~$x_0,x_1,f_2f_4,c$ but not~$a$.
Thanks to~\eqref{c1},~$\pi \mid a^2f_1$; this is impossible since~$\gcd(f_1,f_2f_4)=1$ and~$\pi\nmid a$. Therefore~$a=0$.
\end{proof}


\subsection{Tools from $3$-descent}
\label{subsect:3descent}


The aim is to prove the analogue of Theorem~\ref{thm:main} but in the context of $3$-descent.

\begin{thm}\label{thm:main-3}
Assume hypotheses~\ref{field_k} -- \ref{f1nonct} above. Assume in addition that $f=f_1f_3^3f_5^5$ and that $2\nmid d$. Let $\omega_3$ be the number
of irreducible factors of $f_3$. Let us consider the natural map
\begin{align*}
E(k[t])\to E(k(t))/\lambda^t E'(k(t)).
\end{align*}
\begin{enumerate}
\item\label{item:mainthsmallht-3} When restricted to~$E(k[t])_{< d/4}$, this map is at most $2^{\omega_3}$-to-$1$ (resp. $(3\cdot 2^{\omega_3})$-to-$1$) onto its image
if $\zeta_3\notin k$ (resp. if $\zeta_3\in k$), and omits $0$. In particular,
$$
\#E(k[t])_{< d/4} \leq
 \begin{cases}
 2^{\omega_3}\left(3^{\rk_{\Z} E(k(t))}-1\right) & \text{if $\zeta_3\notin k$} \\
 2^{\omega_3}\left(3^{\frac{1}{2}\rk_{\Z} E(k(t))+1}-3\right) & \text{if $\zeta_3\in k$.}
 \end{cases}
$$
\item\label{item:mainthlargeht-3} When restricted to~$E(k[t])_{d/4\leq \cdot < d/2}$ or $E(k[t])_{\geq d/2}$, it is at most $2^{\omega_3}$-to-$1$ (resp. $(3\cdot 2^{\omega_3})$-to-$1$) onto its image if $\zeta_3\notin k$ (resp. if $\zeta_3\in k$).
In particular,
$$
\# E(k[t])_{\geq d/4} \leq
 \begin{cases}
 2^{\omega_3+1}\cdot 3^{\rk_{\Z} E(k(t))} & \text{if $\zeta_3\notin k$} \\
 2^{\omega_3+1}\cdot 3^{\frac{1}{2}\rk_{\Z} E(k(t))+1} & \text{if $\zeta_3\in k$.}
 \end{cases}
$$
\item\label{item:mainth-3} The set~$E(k[t])$ of all integral points on~$E$ satisfies
\begin{align*}
\#E(k[t])&\leq
 \begin{cases}
 2^{\omega_3}\cdot \left(3^{\rk_{\Z} E(k(t))+1}-1\right) & \text{if $\zeta_3\notin k$} \\
 2^{\omega_3}\cdot \left(3^{\frac{1}{2}\rk_{\Z} E(k(t))+2} -3\right) & \text{if $\zeta_3\in k$.}
 \end{cases}
\end{align*}
\end{enumerate}
\end{thm}

The following Lemma is analogous to (and even easier than) Lemma~\ref{RRlem} and can be proved similarly.

\begin{lem}\label{RRlem-3}
Assume that $f=f_1f_3^3f_5^5$. Then the ring of rational functions on $C_3$ which are regular outside infinity can be explicitly described as
$$
k[t] \oplus k[t]\cdot \frac{y}{f_3f_5^2}.
$$
Moreover, if~$\psi=\beta +\alpha\frac{y}{f_3f_5^2}$
belongs to this ring, then the divisor of its poles satisfies
\begin{align*}
\ord_\infty^\star (\psi)_\infty
=
\begin{cases}
\max\{2\deg_t \beta, 2\deg_t \alpha + d_1+d_3+d_5 \}
&\text{if $2\nmid d$},\\
\max\left\{\deg_t \beta, \deg_t \alpha + \frac{d_1+d_3+d_5}{2}\right\}
&\text{if $2\mid d$,}
\end{cases}
\end{align*}
(note that the rationals appearing in this formula are all integers).
\end{lem}

\begin{lem}[Key Lemma for the $3$-descent]\label{lem:start-3}
Assume that $f=f_1f_3^3f_5^5$. Consider the composition
$$
E(k[t])\to E(k(t))/\lambda^t E'(k(t))\hookrightarrow\Pic(C_3,\mathbb{Q}.D_3)[3].
$$
If two points $P_0=(x_0(t),y_0(t))$ and $P_1=(x_1(t),y_1(t))$ in $E(k[t])$ have the same image by this map, then there exists a rational function
$\psi\in k(C_3)$, regular outside infinity, such that
\begin{equation}
\label{eq:psidef3}
(y_0(t)-y)(y_1(t)+y)=\psi^3.
\end{equation}
Moreover:
\begin{enumerate}
\item\label{item:3descente_f5_exclude} there exist polynomials $a$ and~$b$ in $k[t]$, with~$b\neq 0$, such that
$$
\psi =b+a\frac{y}{f_3}
$$
and we have the relations
\begin{subequations}
\begin{align}
& a^3f_1f_5^5 + 3\frac{ab^2}{f_3} = y_0 - y_1 \label{c1-3} \\
& b^3 + 3a^2b f_1f_3f_5^5= y_0y_1 - f_1f_3^3f_5^5 \label{c0-3}
\end{align}
\end{subequations}
\item\label{item:3descente_if_gcd} if~$\gcd(y_0,f_3) = \gcd(y_1,f_3)$ then~$f_3 \mid a$;
\item\label{item:ifaeq03} if $a=0$ then $(x_1(t),y_1(t)) = (\zeta x_0(t),y_0(t))$ for some~$\zeta\in k$ satisfying~$\zeta^3=1$.
\end{enumerate}
\end{lem}

\begin{proof}
The above composition associates to $(x_0(t),y_0(t))$ the divisor class of the divisor $D_0=\frac{1}{3}\divisor(y_0(t)-y)$
in~$\Pic(C_3,\mathbb{Q}.D_3)[3]$, where we view $y_0(t)-y$ as a rational function in $k(C_3)$. 
Obviously, if~$\zeta_3\in k$, the three distinct points~$(x_0(t),y_0(t))$,~$(\zeta_3 x_0(t),y_0(t))$ and~$(\zeta_3^2 x_0(t),y_0(t))$ map to the same
element of $\Pic(C_3,\mathbb{Q}.D_3)[3]$.  Let $(x_1(t),y_1(t))\in  E(k[t])$ be another point and let $D_1=\frac{1}{3}\divisor(y_1(t)-y)$. 
Suppose that $[D_0]=[D_1]$ in $\Pic(C_3,\mathbb{Q}.D_3)[3]$, or equivalently, $D_0-D_1$ is a principal divisor.  Let $\divisor(\phi)=D_0-D_1$, where $\phi\in k(C_3)$.  Then $\divisor(\phi^3)=3D_0-3D_1=\divisor\left(\frac{y_0(t)-y}{y_1(t)-y}\right)$.  Since $(y_1(t)-y)(y_1(t)+y)=y_1(t)^2-f(t)=x_1(t)^3$, it follows that
$$
(y_0(t)-y)(y_1(t)+y)=\psi^3
$$
for some $\psi\in \kbar(C_3)$.  Since the function on the left-hand side is regular outside infinity, one deduces that $\psi$ is regular outside infinity.

Proof of~\ref{item:3descente_f5_exclude}. According to Lemma~\ref{RRlem-3}, there exist polynomials $\alpha$, $\beta$ such
that~$\psi = \beta+\alpha \frac{y}{f_3f_5^2}$. By identifying the coefficients of~$y$ and~$1$ in the identity \eqref{eq:psidef3}
we obtain:
\begin{align}
& \frac{\alpha^3f_1}{f_5} + 3\frac{\alpha\beta^2}{f_3f_5^2} = y_0 - y_1 \label{local-c1-3} \\
& \beta^3 + 3\alpha^2\beta f_1f_3f_5= y_0y_1 - f_1f_3^3f_5^5 \label{local-c0-3}.
\end{align}
Let us prove that~$f_5^2$ divides~$\alpha$ in several steps. First we show that~$\gcd(f_5,x_i) = \gcd(f_5,y_i) = 1$.
Let~$v$ denotes the valuation associated to an irreducible divisor of~$f_5$. Since~$y_i^2 = x_i^3 + f_1f_3^3f_5^5$, then~$v(y_i) = 1$
(resp~$v(y_i) = 2$) implies~$3v(x_i) = 2$ (resp.~$3v(x_i) = 4$) and~$v(y_i) \geq 3$ implies~$3v(x_i) = 5$; all these lead to contradictions and
thus~$v(y_i) = v(x_i) = 0$. We deduce that~$\gcd(f_5,\beta)=1$ since otherwise, by~\eqref{local-c0-3},
we would have~$\gcd(f_5,y_i) \not= 1$ for at least one~$i$. According to~\eqref{local-c1-3},~$f_5$ divides~$\alpha\beta^2$ and hence~$f_5$
divides~$\alpha$. Going back to~\eqref{local-c1-3}, since~$\gcd(f_5,\beta)= 1$, we deduce that~$f_5^2 \mid\alpha\beta^2$
and thus~$f_5^2 \mid \alpha$. Letting~$b = \beta$ and~$a = \alpha/f_5^2$, item~\ref{item:3descente_f5_exclude} is proved except that~$b\not= 0$.
Assume by contradiction that~$b=0$; then by~\eqref{c0-3}, $f_1 \mid y_0y_1$. But~$f_1$ being nonconstant, up to exchanging~$y_0$ and~$y_1$, there
exists~$\pi$ an irreducible factor of~$f_1$ that divides~$y_0$. This is incompatible with the relation~$y_0^2 - f_1f_3^3f_5^5 = x_0^3$, since
the left-hand term is of $\pi$-valuation equal to~$1$, while the right-hand term must have a $\pi$-valuation divisible by~$3$.

Proof of~\ref{item:3descente_if_gcd}. Let~$\pi$ be an irreducible divisor of~$f_3$. By~\eqref{c1-3}, we know that~$\pi \mid ab^2$ and
thus~$\pi\mid a$ or~$\pi\mid b$. If the latter occurs, then by~\eqref{c0-3}, $\pi \mid y_0y_1$, and hence~$\pi \mid y_0$ and~$\pi \mid y_1$ since~$\gcd(y_0,f_3) = \gcd(y_1,f_3)$. Then~\eqref{c1-3} shows that~$\pi\mid a^3f_1f_5^5$ which implies that~$\pi \mid a$
since~$\gcd(f_3,f_1f_5) = 1$ and~$\pi\mid f_3$. In conclusion, in any case,~$\pi\mid a$ and necessarily~$f_3\mid a$.

Proof of~\ref{item:ifaeq03}. A direct consequence of~\eqref{c1-3}.
\end{proof}

\begin{proof}[Proof of Theorem~\ref{thm:main-3}] \emph{Common starting point}.
In any case, we start with~$P_i = (x_i(t),y_i(t)) \in E(k[t])$, $i=0,1$, having the same image
in~$E(k(t)/\lambda^t E'(k(t))$ by the natural map. Then by Lemma~\ref{lem:start-3}, there exist~$\psi\in k(C_3)$ and~$a,b\in k[t]$
such that
\begin{align}\label{eq_recap_psi_formlas}
&(y_0-y)(y_1+y) = \psi^3
&
&\text{and}
&
&\psi = b + a\frac{y}{f_3}.
\end{align}
We use these two formulas to compute~$\ord_\infty^\star (\psi)_\infty$, but we need to distinguish the cases of \emph{small or large height}. In both cases, assuming that $\gcd(f_3,y_0) = \gcd(f_3,y_1)$, we show that $a=0$.

\emph{Specific argument for~\ref{item:mainthsmallht-3}}. Here we suppose that~$h(P_i) < \frac{d}{4}$. By Lemma~\ref{RRlem-3} (applied with $\alpha=f_3f_5^2$),
we know that~$\ord_\infty^\star (y_i(t)\pm y)_\infty = \max\left\{2\deg y_i, d\right\}$. Since~$h(P_i) < \frac{d}{4}$, we have~$2\deg y_i < \frac{3d}{2}$
and thus~$\ord_\infty^\star (y_i(t)\pm y)_\infty < \frac{3d}{2}$. Therefore~$\ord_\infty^\star (y_0(t)-y)(y_1(t)+y)_\infty < 3d$ and~$\ord_\infty^\star (\psi)_\infty < d$,
by the left-hand side of~\eqref{eq_recap_psi_formlas}.
On the other hand, according to the right-hand side, one has~$\ord_\infty^\star (\psi)_\infty = \max\left\{2\deg b, 2\deg a + d_1+d_3+5d_5\right\}$ (Lemma~\ref{RRlem-3} again, applied with $\alpha=af_5^2$).
Hence~$2\deg a + d_1+d_3+5d_5 < d$ which means that~$\deg a < d_3$.

If moreover~$\gcd(f_3,y_0) = \gcd(f_3,y_1)$ then by Lemma~\ref{lem:start-3}~\ref{item:3descente_if_gcd}, we know that~$f_3\mid a$
and thus~$a=0$.

\emph{Specific argument for~\ref{item:mainthlargeht-3}}. Here we suppose that~$h(P_i) \geq \frac{d}{4}$, and in particular~$h(P_i) > \frac{d}{6}$.
Recall that if this holds then~$h(P_i) = \frac{\deg y_i}{3} = \frac{\deg x_i}{2}$ and necessarily~$3\mid\deg y_i$.

Let~$P_i = (x_i(t),y_i(t)) \in E(k[t])_{\geq d/4}$, $i=0,1$, be two
integral points having the same image in~$E(k(t)/\lambda^t E'(k(t))$ by the natural map. By Lemma~\ref{lem:start-3}, there exists~$\psi\in k(C_3)$
such that~$(y_0-y)(y_1+y) = \psi^3$.
This equality permits us to compute~$\ord_\infty^\star (\psi)_\infty$; indeed:
\begin{align*}
  \ord_\infty^\star (y_i\pm y)_\infty &= \max\left\{2\deg y_i, d\right\} && \text{by Lemma~\ref{RRlem-3}}\\
                                &= 2\deg y_i && \text{since~$h(P_i) = \frac{\deg y_i}{3} \geq \frac{d}{4}$}
\end{align*}
and then
\begin{align}\label{3descente_ord_psi_1}
&\ord_\infty^\star ((y_0-y)(y_1+y))_\infty = 2\left(\deg y_0 + \deg y_1\right)
&
&\Longrightarrow
&
&\ord_\infty^\star (\psi)_\infty = 2\left(\frac{\deg y_0}{3} + \frac{\deg y_1}{3}\right).
\end{align}
In particular this order is an even integer since~$3$ divides~$\deg y_i$.

On the other hand, taking into account the fact that the function~$\psi$ can be written~$\psi = b + a \frac{y}{f_3}$ with~$a,b\in k[t]$,
we also have
\begin{equation}\label{3descente_ord_psi_2}
\ord_\infty^\star (\psi)_\infty = \max\left\{2\deg b, 2\deg a + d_1+d_3+5d_5\right\}
\end{equation}
(Lemmas~\ref{lem:start-3} and~\ref{RRlem-3} again). Since~$2\nmid d$, only the first integer inside the brackets is even, hence by
comparing formulas~\eqref{3descente_ord_psi_1} and~\eqref{3descente_ord_psi_2} we deduce that:
\begin{align}\label{jaipudidee}
2\deg b &= 2\left(\frac{\deg y_0}{3} + \frac{\deg y_1}{3}\right)
&
&\text{and}
&
&2\deg a + d_1+d_3+5d_5 < 2\deg b.
\end{align}
Adding~$\deg a$ to the last inequality leads to~$3\deg a + d_1 + 5d_5 < \deg a + 2\deg b - d_3$, and thus by~\eqref{c1-3}
one has~$\deg a + 2\deg b - d_3 = \deg(y_0-y_1)$.

Let us now assume that the condition~$h(P_1)\leq h(P_0) < 2h(P_1)$ holds, which applies to any appropriately ordered pair of points in each of the two regions considered (one for obvious reasons, and the other one by Davenport's inequality). Then
\begin{align*}
  \deg a &= d_3 - 2\deg b - \deg(y_0-y_1)\\
         &= d_3 - 2\left(\frac{\deg y_0}{3} + \frac{\deg y_1}{3}\right) + \deg(y_0-y_1) &&\text{by~\eqref{jaipudidee}}\\
         &\leq d_3 - 2\left(\frac{\deg y_0}{3} + \frac{\deg y_1}{3}\right) + \deg y_0 &&\text{since~$h(P_1) \leq h(P_0)$}\\
         &< d_3 &&\text{since~$h(P_0) < 2h(P_1)$.}
\end{align*}
If moreover~$\gcd(f_3,y_0) = \gcd(f_3,y_1)$, then by Lemma~\ref{lem:start-3}~\ref{item:3descente_if_gcd}, we know that~$f_3\mid a$
and thus~$a=0$.

\emph{Common concluding argument}. Let us define the map
$$
E(k[t]) \longrightarrow E(k(t))/\lambda^t E'(k(t)) \times \{0,1\}^{\omega_3},
$$
where the first coordinate is none other than the natural map, and the second coordinate is the map
$$
(x(t),y(t)) \mapsto (v_\pi(\gcd(f_3,y))_{\pi},
$$
where $\pi$ runs through the set of irreducible factors of~$f_3$.
Using Lemma~\ref{lem:start-3}~\ref{item:ifaeq03}, all the previous specific arguments lead to the following fact: when restricted
to~$E(k[t])_{<d/4}$, or~$E(k[t])_{d/4\leq \cdot < d/2}$, or~$E(k[t])_{\geq d/2}$, the preceding map is $1$-to-$1$ if~$\zeta_3\not\in k$ and
$3$-to-$1$ if~$\zeta_3\in k$. Finally, we note that
$$
\dim_{\F_3} E(k(t))/\lambda^t E'(k(t))\leq \rk_{\Z} E(k(t)).
$$

In the case when $\zeta_3\in k$, the curves $E$ and $E'$ are isomorphic, the groups $E(k(t))/\lambda^t E'(k(t))$ and $E'(k(t))/\lambda E(k(t))$ are isomorphic, and hence from the exact sequence of Proposition~\ref{prop:3descent} we are able to deduce that
$$
\dim_{\F_3} E(k(t))/\lambda^t E'(k(t))=\frac{1}{2}\rk_{\Z} E(k(t)).
$$
The result follows.
\end{proof}


\section{Integral points of small height}


We first prove a Lemma which allows us to transport computations of integral points to a convenient Weierstrass model.

\begin{lem}
\label{lem:changeofcurve}
Let $f\in k[t]$ be a polynomial of degree $d$ and let
\begin{align*}
E: y^2&=x^3+f(t),\\
E'': y^2&=x^3+t^{6\lceil d/6\rceil}f(1/t).
\end{align*}
The bijection
\begin{align*}
E(k(t))&\to E''(k(t))\\
(x(t),y(t))&\to (t^{2\lceil d/6\rceil}x(1/t),t^{3\lceil d/6\rceil}y(1/t))
\end{align*}
induces a bijection between the sets of integral points $E(k[t])_{\leq \lceil d/6\rceil}$ and $E''(k[t])_{\leq \lceil d/6\rceil}$ that preserves the canonical height of the points.  Let $C_i, C_i'', D_i, D_i''$, $i=2,3$, be the associated curves and divisors.  Then the isomorphisms
\begin{align*}
C_2&\to C_2''\\
(t,x)&\mapsto (t^{-1},t^{-2\lceil d/6\rceil}x)
\end{align*}
and
\begin{align*}
C_3&\to C_3''\\
(t,y)&\mapsto (t^{-1},t^{-3\lceil d/6\rceil}y),
\end{align*}
and the bijection above, induce commutative diagrams for $i=2,3$:
\begin{equation*}
\begin{CD}
E(k(t)) @>>> \Pic(C_i, \mathbb{Q}.D_i)[i]\\
@V VV @VV V \\
E''(k(t)) @>>> \Pic(C_i'', \mathbb{Q}.D_i'')[i]. \\
\end{CD}
\end{equation*}
\end{lem}

\begin{proof}
If we base-change the curve $E$ by the automorphism $t\mapsto 1/t$ of $k(t)$, we obtain the elliptic curve $E''$ defined by the equation $y^2=x^3+f(1/t)$. The map $(x(t),y(t))\mapsto (x(1/t), y(1/t))$ induces a bijection $E(k(t))\to E''(k(t))$, whose inverse is given by the same recipe. By the (obvious) change of coordinates $(x,y)\mapsto (t^{2\lceil d/6\rceil}x,t^{3\lceil d/6\rceil}y)$, $E''$ can be defined by the Weierstrass equation
$$
y^2=x^3+t^{6\lceil d/6\rceil}f(1/t).
$$

The main interest of considering this equation is that $t^{6\lceil d/6\rceil}f(1/t)$ is again a polynomial in $t$. We shall now consider integral points on $E''$ with respect to this specific Weierstrass equation.

The explicit bijection $E(k(t))\to E''(k(t))$ described in the statement of the Lemma is obtained by composing the bijection $t\mapsto 1/t$ with the previous change of coordinates. Similarly, $C_i\simeq C_i''$ as $k$-curves, and the bijection is described by the same recipe. The diagram comparing the $2$-descent maps on $E$ and $E''$ (resp. the $3$-descent maps) commutes for obvious reasons.

Denoting by $\Emin\to \PP^1_k$ (resp. $\Emin''\to \PP^1$) the minimal regular model of $E$ (resp. $E''$), we have a Cartesian square
$$
\begin{CD}
\Emin'' @>>> \Emin \\
@VVV @VVV \\
\PP^1 @>>t\mapsto 1/t> \PP^1 \\
\end{CD}
$$
and in particular, $\Emin$ and $\Emin''$ are isomorphic $k$-surfaces. Since the canonical (N{\'e}ron-Tate) height can be computed via intersection theory between divisors on these surfaces \cite[Theorem~8.6]{Shioda2}, the bijection $E(k(t))\to E''(k(t))$ preserves the canonical height. Alternatively, one can check that the heights are preserved by a direct computation, applying Theorem~\ref{thh}.

Finally, we claim that the bijection $E(k(t))\to E''(k(t))$ induces by restriction a bijection between  $E(k[t])_{\leq \lceil d/6\rceil}$ and $E''(k[t])_{\leq \lceil d/6\rceil}$. It suffices to check that $E(k(t))\to E''(k(t))$ sends $E(k[t])_{\leq \lceil d/6\rceil}$ to $E''(k[t])_{\leq \lceil d/6\rceil}$, and similarly for the inverse bijection. This is elementary: if $x(t)$ has degree $\leq 2\lceil d/6\rceil$ then $t^{2\lceil d/6\rceil}x(1/t)$ is a polynomial in $t$, and likewise if $y(t)$ has degree $\leq 3\lceil d/6\rceil$ then $t^{3\lceil d/6\rceil}y(1/t)$ is a polynomial. The other way around is left to the reader.
\end{proof}

Since the canonical height is a quadratic form on $E(\kbar(t))$, we can define an associated bilinear pairing, the height pairing, by 
\begin{equation*}
\langle P,Q \rangle=\hat{h}(P+Q)-\hat{h}(P)-\hat{h}(Q).
\end{equation*}
This pairing gives $E(\kbar(t))/E(\kbar(t))_{\tors}$ the structure of a positive-definite lattice, called the \emph{Mordell-Weil lattice} of $E$. In our setting, the elliptic curves $E$ have complex multiplication by $\mathbb{Z}[\zeta_3]$, where $\zeta_3$ is a primitive third root of unity. Then $\mathbb{Z}[\zeta_3]$ acts on the Mordell-Weil lattice and since
\begin{equation*}
\hat{h}(\alpha P)=N(\alpha)\hat{h}(P), \quad \forall \alpha\in \mathbb{Z}[\zeta_3],
\end{equation*}
(where $N=N_{\Q(\zeta_3)/\Q}$ is the norm), we deduce that $\langle \alpha P, \alpha Q\rangle =N(\alpha)\langle P,Q\rangle$.

We first prove some facts about lattices that will be useful in studying $2$-descent and $3$-descent maps. This leads to an alternative approach to some of the results of Section \ref{sec:ub}.

For a lattice $L$ and $v\in L$, we define the norm of $v$ to be $N(v)=\langle v,v\rangle$ and let $L_\alpha=\{v\in L\mid N(v)=\alpha\}$ be the set of elements of norm $\alpha$. If $L$ also has the structure of a $\mathbb{Z}[\zeta_3]$-module such that $N(\alpha v)=N(\alpha)N(v)$ for all $\alpha\in \mathbb{Z}[\zeta_3], v\in L$, we will call $L$ a $\mathbb{Z}[\zeta_3]$-lattice. We now study the maps $L\to L/2L$ and $L\to L/\sqrt{-3}L$ (when $L$ is a $\mathbb{Z}[\zeta_3]$-lattice), restricted to certain sets of elements.

\begin{lem}
\label{Ldesc}
Let $L$ be a lattice and let $\mu$ be the minimum norm of $L$. Then the canonical map
\begin{align*}
L_{\mu}\to L/2L
\end{align*}
is $2$-to-$1$ and omits $0$ in its image. Suppose further that $L$ is a $\mathbb{Z}[\zeta_3]$-lattice.  Then
\begin{align*}
v\equiv \sqrt{-3}v \pmod{2L}
\end{align*}
for all $v\in L$, and in particular the canonical map
\begin{align*}
\sqrt{-3}L_{\mu}\to L/2L
\end{align*}
is $2$-to-$1$ and omits $0$ in its image. Moreover, the canonical map
\begin{align*}
L_{\mu}\to L/\sqrt{-3}L
\end{align*}
is $3$-to-$1$ and omits $0$ in its image. 
\end{lem}

\begin{proof}
The minimum norm of $2L$ is $4\mu$. If $v,w\in L_\mu$, then by the triangle inequality, $N(v-w)\leq 4\mu$, with equality if and only if $w=-v$. Thus, $v\equiv w\pmod{2L}, v,w\in L_\mu,$ implies that $w=\pm v$. Since $v$ and $-v$ map to the same element of $L/2L$, the first result follows immediately.

Suppose now that $L$ is a $\mathbb{Z}[\zeta_3]$-lattice. Since $v-\sqrt{-3}v=2(-\zeta_3 v)\in 2L$, we immediately find $v\equiv \sqrt{-3}v \pmod{2L}$ for all $v\in L$. Then the statement for the map $\sqrt{-3}L_{\mu}\to L/2L$ follows immediately from the same statement for the map $L_\mu\to L/2L$.

We now prove the last statement. First, we note that since $\sqrt{-3}$ divides $1-\zeta_3$ and $1-\zeta_3^2$ in $\mathbb{Z}[\zeta_3]$, we have
\begin{align}
\label{-3L}
v\equiv \zeta_3 v\equiv \zeta_3^2 v\pmod{\sqrt{-3}L}
\end{align}
for all $v\in L$.

Let $v,w\in L_\mu$, $v\neq w$. Then, by definition of $\mu$, we have
\begin{align*}
N(v-w)=2\mu-2\langle v,w\rangle\geq \mu. 
\end{align*}
After possibly replacing $w$ by $-w$, it follows that if $v\neq \pm w$, then $|\langle v,w\rangle|\leq \frac{1}{2}\mu$. 

Suppose now that $v-w=\sqrt{-3}u\in \sqrt{-3}L, v\neq w$. Then $N(v-w)=N(\sqrt{-3}u) \geq 3\mu$, and by the same calculation as above, $\langle v,w\rangle\leq -\frac{1}{2}\mu$.  So we must have either equality $\langle v,w\rangle= -\frac{1}{2}\mu$ or $v=-w$. If $v\neq - w$, then
\begin{align*}
N(v+2w)=N(v)+4\langle v,w\rangle+4N(w)=3\mu 
\end{align*}
and
\begin{align*}
v+2w\equiv v+2\zeta_3 v\equiv v+2\zeta_3^2 v \pmod{2L}. 
\end{align*}
We have $v+2w=(v-w)+3w=\sqrt{-3}u+3w=\sqrt{-3}(u-\sqrt{-3}w)$ and so $N(v+2w)=N(\sqrt{-3}(u-\sqrt{-3}w))=3N(u-\sqrt{-3}w)$, which is equal to $3\mu$. It follows that $N(u-\sqrt{-3}w)=\mu$ and $v+2w=\sqrt{-3}(u-\sqrt{-3}w)\in \sqrt{-3}L_\mu$. We also have  $v+2\zeta_3 v=(1+2\zeta_3) v=\sqrt{-3}v\in \sqrt{-3}L_\mu$ and similarly $v+2\zeta_3^2 v\in \sqrt{-3}L_\mu$. Since $\sqrt{-3}L_{\mu}\to L/2L$ is $2$-to-$1$, we conclude that $w=\zeta_3 v$ or $w=\zeta_3^2 v$. 

Finally, we show that $w=-v$ is impossible. Suppose $2v=\sqrt{-3}u\in \sqrt{-3}L$. Then taking norms of both sides gives $N(u)=\frac{4}{3}\mu$. We also have
\begin{align*}
2v+2\zeta_3^2u=(\sqrt{-3}+2\zeta_3^2)u=-u
\end{align*}
and so
\begin{align*}
N(2v+2\zeta_3^2u)=4N(v+\zeta_3^2u)=N(-u)=\frac{4}{3}\mu,
\end{align*}
a contradiction since this would imply $N(v+\zeta_3^2u)=\frac{1}{3}\mu$. Therefore $v\not\equiv -v\pmod{\sqrt{-3}L}$. Combined with the above and noting \eqref{-3L}, we see that $L_{\mu}\to L/\sqrt{-3}L$ is $3$-to-$1$ and obviously omits $0$ in its image.

\end{proof}

\begin{lem}
\label{mu}
Assume hypotheses~\ref{field_k} -- \ref{f12} and that $f$ is not a perfect power in $\kbar[t]$. Let $\mu$ be the minimal (nonzero) canonical height of a point in $E(\kbar(t))$ and let
\begin{align*}
E_{\min}&=\{P\in E(\kbar(t))\mid \hat{h}(P)=\mu\}\\
E'_{\min}&=\{P\in E'(\kbar(t))\mid \hat{h}(P)=\mu\}.
\end{align*}
Let $\tilde{\phi}_2$ and $\tilde{\phi}_3$ denote the canonical lifts of the descent maps $\phi_2$ and $\phi_3$ (defined in Prop.~\ref{prop:2descent} and Prop.~\ref{prop:3descent}) to $E_{\min}$, and let $A_2$ and $A_3$ be the images of $\tilde{\phi}_2$ and $\tilde{\phi}_3$, respectively. Similarly, define $\tilde{\phi}_3'$ and $A_3'$ associated to the descent map $\phi_3'$ and $E'_{\min}$.  Then there are $2$-to-$1$ maps
\begin{align*}
\{P\in E_{\min}\mid x(P)\in k(t)\}&\to A_2(k),\\
\{P\in \sqrt{-3} E_{\min}\mid x(P)\in k(t)\}&\to A_2(k),
\end{align*}
and $3$-to-$1$ maps
\begin{align*}
\{P\in E_{\min}\mid y(P)\in k(t)\}&\to A_3(k),\\
\{P\in \sqrt{-3} E_{\min}\mid y(P)\in k(t)\}&\to A_3'(k).
\end{align*}
\end{lem}

\begin{proof}

Note that if $(x_0,y_0)\in E_{\min}$, then $\tilde{\phi}_2$ sends $(x_0,\pm y_0)$ to the same element, and $\tilde{\phi}_3$ sends $(\zeta_3^i x_0,y_0)$, $i=0,1,2$, to the same element. Since $\tilde{\phi}_2$ and $\tilde{\phi}_3$ are $\Gal(\kbar/k)$-equivariant maps, it is immediate that if $\tilde{\phi}_2$ is $2$-to-$1$, then it induces a $2$-to-$1$ map $\{(x_0,y_0)\in E_{\min}\mid x_0\in k(t)\}\to A_2(k)$, and similarly if $\tilde{\phi}_3$ is $3$-to-$1$, it induces a $3$-to-$1$ map $\{(x_0,y_0)\in E_{\min}\mid y_0\in k(t)\}\to A_3(k)$.

Let $L$ be the Mordell-Weil lattice associated to $E(\kbar(t))$.  The elliptic curve $E$ has complex multiplication by the ring of integers $\O_{\mathbb{Q}(\zeta_3)}=\mathbb{Z}[\zeta_3]$ of $\mathbb{Q}(\zeta_3)=\mathbb{Q}(\sqrt{-3})$, and this induces a $\mathbb{Z}[\zeta_3]$-lattice structure on $L$. According to \cite[Theorem~8.7]{Shioda2}, the identity component of the N{\'e}ron model of $E$ has no torsion; since $f$ is not a perfect power, it follows easily from the types of singular fibers of $E$ that the torsion subgroup of $E(\kbar(t))$ is trivial, and we have an isomorphism of $\mathbb{Z}[\zeta_3]$-modules $E(\kbar(t))\to L$. 

Let $\psi$ be the isomorphism (over $k(\sqrt{-3})$) $\psi:E\to E'$, $(x,y)\mapsto (-3x, -3\sqrt{-3}y)$. Then it is easily verified that $\lambda^t(\psi(P))=\sqrt{-3}P$, and so $\lambda^t(E'(\kbar(t)))=\sqrt{-3}E(\kbar(t))$.  It follows that we have commutative diagrams
\begin{equation*}
\begin{CD}
E_{\min} @>>> E(\kbar(t))/2E(\kbar(t)) \\
@V VV @VV  V \\
L @>>> L/2L \\
\end{CD}
\end{equation*}
and
\begin{equation*}
\begin{CD}
E_{\min} @>>> E(\kbar(t))/\lambda^t(E'(\kbar(t))) \\
@V VV @VV  V \\
L @>>> L/\sqrt{-3}L. \\
\end{CD}
\end{equation*}

Then by Lemma~\ref{Ldesc}, the maps $\tilde{\phi}_2$ and $\tilde{\phi}_3$ are $2$-to-$1$ and $3$-to-$1$, respectively, proving the statements for $E_{\min}$ (using the previous remarks).

Since $\sqrt{-3}\equiv 1\pmod{2\mathbb{Z}[\zeta_3]}$, we have a commutative diagram 
\begin{equation*}
\begin{CD}
E_{\min} @>>> E(\kbar(t))/2E(\kbar(t)) \\
@V\sqrt{-3} VV @VV id V \\
\sqrt{-3}E_{\min} @>>> E(\kbar(t))/2E(\kbar(t)) \\
\end{CD}
\end{equation*}
proving that there is a $2$-to-$1$ map $\{P\in \sqrt{-3} E_{\min}\mid x(P)\in k(t)\}\to A_2(k)$. Finally, we note that $\lambda^t:E'_{\min}\to \sqrt{-3}E_{\min}$ is bijective, and the $y$-coordinate of $\lambda^t$ is a rational function in $y$ over $k$. It follows that the composition
\begin{align*}
\{P\in \sqrt{-3} E_{\min}\mid y(P)\in k(t)\}\stackrel{(\lambda^t)^{-1}}{\rightarrow} \{P\in E'_{\min}\mid y(P)\in k(t)\}\to A_3'(k)
\end{align*}
is $3$-to-$1$.
\end{proof}

We study in more detail integral points of small height when $F(t,u)$ is squarefree up to a power of a single linear form (equivalently, up to a change of variable, the inhomogeneous polynomial $f(t)$ is squarefree). In this case, we relate integral points of small height with torsion points lying on certain (generalized) Brill-Noether loci.

Classically, for a curve $C$ and positive integer $d$, one studies the Brill-Noether locus  consisting of the classes in $\Pic^d(C)$ that can be represented by effective divisors of degree $d$. Choosing a base divisor $D_0$ of degree $d$, one may embed the Brill-Noether locus in $\Pic^0(C)$ via the mapping $[D]\mapsto [D-D_0]$. In our setting, suppose that $f(t)$ is squarefree, and let $C$ be one of the curves $C_2$ or $C_3$ arising in the $2$-descent and $3$-descent maps. In either case, we let $D_\infty$ denote the reduced divisor with support $t^{-1}(\infty)$. Let $P=(x_0,y_0)\in E(\kbar[t])$ be an integral point of canonical height $n=\hat{h}(P)=h(P)\in \frac{1}{6}\mathbb{Z}$. Then $\deg x_0\leq 2n$ and $\deg y_0\leq 3n$. For a $\mathbb{Q}$-divisor $D$, write $D\geq 0$ if $D$ is effective. We shall see in the proof of Theorem~\ref{th2desc} that on $C_2$,
\begin{align*}
\frac{1}{2}\divisor(x_0-x)+ \frac{3n}{\deg D_\infty}D_\infty\geq 0
\end{align*}
and on $C_3$,
\begin{align*}
\frac{1}{3}\divisor(y_0-y)+ \frac{2n}{\deg D_\infty}D_\infty\geq 0.
\end{align*}

Therefore, if we define
\begin{align*}
W_n(C_2)[2]&=\{[D]\in \Pic(C_2, \mathbb{Q}.D_2)[2] \mid D+n/(\deg D_\infty)D_\infty\geq 0\}\\
W_n(C_3)[3]&=\{[D]\in \Pic(C_3, \mathbb{Q}.D_3)[3]\mid D+n/(\deg D_\infty)D_\infty\geq 0\},
\end{align*}
then the $2$ and $3$-descent maps induce maps
\begin{align*}
\left\{P=(x_0,y_0)\in E(\kbar(t))\mid x_0\in k[t], \hat{h}(P)\leq n\right\}&\to W_{3n}(C_2)[2]\\
\left\{P=(x_0,y_0)\in E(\kbar(t))\mid y_0\in k[t], \hat{h}(P)\leq n\right\}&\to W_{2n}(C_3)[3].
\end{align*}

\subsection{Arguments using $2$-descent}

We assume throughout that $f(t)$ is squarefree.
\begin{thm}
\label{th2desc}
Let $n\in\mathbb{Q}, n>0$. Then there is a map
\begin{align*}
\phi_2:\left\{P=(x_0,y_0)\in E(\kbar(t))\mid x_0\in k[t], 0<\hat{h}(P)\leq n\right\}&\to W_{3n}(C_2)[2]\setminus \{0\}\\
P=(x_0,y_0)&\mapsto \left[\frac{1}{2}\divisor(x_0-x)\right]
\end{align*}
\begin{enumerate}
\item If $n<d/3$, the map $\phi_2$ is $2$-to-$1$ onto its image.
\item If $n\leq d/6+1/3$, the map $\phi_2$ is $2$-to-$1$ and surjective.\label{th2descii}
\end{enumerate}
\end{thm}

\begin{rmq}
\label{smallcan}
For $0<n\leq d/6+1/3$, there are points of canonical height $n$ on $E$ only if:
\begin{align*}
n=\frac{d}{6}, \quad &\text{if }d\equiv 0,2,3\pmod{6},\\
n=\frac{d}{6}, \frac{d}{6}+\frac{1}{3},  \quad &\text{if }d\equiv 4\pmod{6},\\
n=\frac{d+1}{6}, \quad &\text{if }d\equiv 5\pmod{6}.
\end{align*}
In particular, there are no such points if $d\equiv 1\pmod{6}$.
\end{rmq}

We note that if $3\nmid d$ and $P=(x_0,y_0)\in E(\kbar[t])$, then $3\hat{h}(P)=3h(P)=\deg y_0$.  

\begin{proof}[Proof of Theorem~\ref{th2desc}]
Let $P=(x_0(t),y_0(t))\in E(\kbar(t))$, with $x_0\in k[t]$ and $0<h(P)=\hat{h}(P)\leq n$.  It is elementary that $h(P)=\frac{1}{6}\max \{d, 3\deg x_0\}$.  By Proposition~\ref{prop:2descent}, we have $[D]\in \Pic(C_2,\mathbb{Q}.D_2)[2]$, where $D=\frac{1}{2}\divisor(x_0(t)-x)$.  By Lemma~\ref{RRlem},
\begin{align*}
\ord_\infty D=\frac{1}{2}\max \{d, 3\deg x_0\}=3h(P)\leq 3n=3n/\deg D_\infty & \text{ if }3\nmid d,\\
\max_{1\leq i\leq 3}\ord_{\infty_i} D=\frac{1}{6}\max \{d, 3\deg x_0\}=h(P)\leq 3n/\deg D_\infty & \text{ if }3\mid d.
\end{align*}

Thus, $D+3n/(\deg D_\infty)D_\infty$ is effective and $\phi_2$ is well-defined.   If $n<d/3$, then $P\in  E(\kbar[t])_{<d/3}$, and so $\phi_2$ is $2$-to-$1$ onto its image by the proof of Theorem~\ref{thm:main}~\ref{item:mainthfeqf1}.

Suppose now that $n\leq \frac{d}{6}+\frac{1}{3}$.  We first claim that $\phi_2$ is $2$-to-$1$ onto its image. If $d\geq 3$, then this implies that $n<\frac{d}{3}$ and $\phi_2$ is  $2$-to-$1$ from the above.  If $d=2$, then it is easy that $0<h(P)\leq \frac{2}{3}$ implies that $h(P)=\frac{1}{3}$, and the claim follows taking $n=\frac{1}{3}<\frac{d}{3}$.  If $d=1$ then $E(\kbar(t))$ is trivial and the claim is vacuously true. 

It remains to show that $\phi_2$ is surjective.   Let $[D]\in W_{3n}(C_2)[2]\setminus\{0\}$ so that 
\begin{align*}
D_0:=D+3n/(\deg D_\infty)D_\infty\in \Div(C_2,\mathbb{Q}.D_2) 
\end{align*}
is effective.  Then
\begin{align*}
2D=2D_0-\frac{6n}{\deg D_\infty}D_\infty
\end{align*}
is principal and $6n\leq d+2<d+3$.  It follows from Lemma~\ref{RRlem} that there exists a polynomial $x_0(t)\in k[t]$ with $\deg x_0\leq 2n$ and $a\in k$ such that $\divisor(x_0(t)+ax)=2D$.  If $a=0$, then this implies that $x_0(t)$ is a perfect square (up to a constant) and $D$ itself is principal, contradicting $[D]\neq 0$.  Then $a\neq 0$ and replacing $x_0$ by $-ax_0$, we obtain a polynomial $x_0(t)\in k[t]$ with $\divisor(x_0(t)-x)=2D$.  This implies that there exists $y_0\in \kbar[t]$ such that $y_0^2=x_0^3+f$.  Moreover, using Lemma~\ref{RRlem} to compute the poles of $x_0(t)-x$, we find that $6n\geq d$. Let $P=(x_0,y_0)$.  Then $h(P)=\max\{\frac{d}{6}, \frac{1}{2}\deg x_0\}\leq n$ and $\phi_2(P)=[D]$. So $\phi_2$ is surjective as well.
\end{proof}

An elementary calculation using the Riemann-Hurwitz formula yields the canonical divisor (and the genus of $C_2$, already given in \eqref{eq:genusofC2}).

\begin{lem}
\label{canlem}
Let $K_{C_2}$ be the canonical divisor of $C_2$.  Then
\begin{align*}
K_{C_2}\sim 
\begin{cases}
(2d-4)\infty \quad &\text{if }3 \nmid d,\\
\left(\frac{2d}{3}-2\right)D_\infty, \quad &\text{if } 3\mid d.
\end{cases}
\end{align*}
\end{lem}

In certain cases, the $2$-torsion in $W_n(C_2)$ can be given a more useful characterization. Let $C$ be a nonsingular projective curve of genus $g$. Recall that a theta characteristic of $C$ is an element $\mathcal{L}\in \Pic(C)$ such that $\mathcal{L}^{\otimes 2}\cong \omega_C$, the canonical bundle of $C$. Abusing terminology, we will also call a divisor $D$ a theta characteristic if $\O(D)$ is a theta characteristic. The theta characteristic is called even or odd if the dimension of the space of global sections, $h^0(\mathcal{L})$, is even or odd, respectively. We say that $\mathcal{L}$ is an effective theta characteristic if $h^0(\mathcal{L})>0$, and that $\mathcal{L}$ is a vanishing theta characteristic if $h^0(\mathcal{L})>1$. Over an algebraically closed field of characteristic $\neq 2$, it is well-known that there are precisely $2^{g-1}(2^{g-1}+1)$ even theta characteristics and $2^{g-1}(2^{g-1}-1)$ odd theta characteristics.  Note that if $\mathcal{L}$ is a theta characteristic, then $\mathcal{L}'\in \Pic(C)$ is a theta characteristic if and only if $\mathcal{L}'\otimes \mathcal{L}^{-1}\in \Pic(C)[2]$ is a $2$-torsion element. In particular, there are $2^{2g}$ theta characteristics over an algebraically closed field of characteristic $\neq 2$.

\begin{lem}
\label{lemWC2}
Let
\begin{equation*}
g=g(C_2)= \left\{
\begin{array}{ll}
d-1 & \text{if }3 \nmid d,\\
d-2 & \text{if}~ 3\mid d 
\end{array}\right.
\end{equation*}
and
\begin{align*}
\frac{1}{2}K_{C_2}:= 
\begin{cases}
(d-2)\infty \quad &\text{if }3 \nmid d,\\
\left(\frac{d}{3}-1\right)D_\infty, \quad &\text{if } 3\mid d.
\end{cases}
\end{align*}
\begin{enumerate}
\item
If $3\nmid d$ then
\begin{align*}
W_g(C_2)[2]&=\Pic(C_2)[2].
\end{align*}
\item If $d\not\equiv 3\pmod{6}$ then
\begin{align*}
W_{g-1}(C_2)[2]&\to \{\text{effective theta characteristics in $\Pic(C_2)$}\}\\
[D]&\mapsto \left[D+\frac{1}{2}K_{C_2}\right]
\end{align*}
is a bijection.
\item If $d\equiv 3\pmod{6}$ then there is a bijection
\begin{align*}
W_{g+\frac{1}{2}}(C_2)[2]&\to \left(\Pic(C_2, \mathbb{Q}.D_2)[2]\setminus \Pic(C_2)[2]\right)\cup\{\text{effective theta characteristics in $\Pic(C_2)$}\}.
\end{align*}
\end{enumerate}
\end{lem}

\begin{proof}
Suppose first that $3\nmid d$. Then $\Pic(C_2, \mathbb{Q}.D_2)=\Pic(C_2)$.  For any line bundle $\mathcal{L}$, by the Riemann-Roch formula the condition $h^0(\mathcal{L}\otimes \O_{C_2}(gD_\infty))>0$ is always satisfied, and so clearly $W_g(C_2)[2]=\Pic(C_2)[2]$.

Suppose now that $d\not\equiv 3\pmod{6}$. Then again $\Pic(C_2, \mathbb{Q}.D_2)=\Pic(C_2)$. Note that $\frac{1}{2}K_{C_2}=\frac{g-1}{\deg D_\infty}D_\infty$ and by Lemma~\ref{canlem}, $\frac{1}{2}K_{C_2}$ is a theta characteristic. Then the claimed bijection between $W_{g-1}(C_2)[2]$ and effective theta characteristics of $C_2$ follows immediately from the definitions.

Finally, suppose that $d\equiv 3\pmod{6}$. Let $W_1=W_{g+\frac{1}{2}}(C_2)[2]\cap \Pic(C_2)$ and $W_2=W_{g+\frac{1}{2}}(C_2)[2]\cap \left(\Pic(C_2, \mathbb{Q}.D_2)\setminus \Pic(C_2)\right)$. We first consider $W_1$. In this case, since $\left\lfloor\frac{g+\frac{1}{2}}{\deg D_\infty}\right\rfloor=\left\lfloor\frac{d}{3}-\frac{1}{2}\right\rfloor=\frac{d}{3}-1=\frac{g-1}{\deg D_\infty}$, we see that for a $\mathbb{Z}$-divisor $D$, $D+\frac{g+\frac{1}{2}}{\deg D_\infty}D_\infty$ is effective if and only if $D+\frac{g-1}{\deg D_\infty}D_\infty$. Then by the same proof as in case (ii), $W_1$ is in bijection with effective theta characteristics in $\Pic(C_2)$.

Now let $[D]\in \Pic(C_2, \mathbb{Q}.D_2)[2]\setminus \Pic(C_2)[2]$. Then we may write the principal divisor $2D$ as $2D=2D'-\epsilon_1\infty_1-\epsilon_2\infty_2-\epsilon_3\infty_3$, where $D'$ is a $\mathbb{Z}$-divisor and $\epsilon_i\in \{0,1\}$, $i=1,2,3$, not all $0$. Since $\deg 2D=0$, looking mod $2$ we see that $\epsilon_i=0$ for some unique $i$, which after reindexing we can assume is $\infty_3$. It follows that $\infty_1+\infty_2$ and $\infty_3$ are $k$-rational. We also have
\begin{align*}
\deg\left(D'+\left(\frac{d}{3}-1\right)D_\infty\right)=g,
\end{align*}
where $D'+\left(\frac{d}{3}-1\right)D_\infty$ is a $\mathbb{Z}$-divisor. Then by Riemann-Roch, $D'+\left(\frac{d}{3}-1\right)D_\infty\sim E$ for some effective divisor $E$.  Then 
\begin{align*}
D+\frac{g+\frac{1}{2}}{\deg D_\infty}D_\infty=D+\left(\frac{d}{3}-\frac{1}{2}\right)D_\infty=D'+\frac{1}{2}\infty_3+\left(\frac{d}{3}-1\right)D_\infty=E+\frac{1}{2}\infty_3
\end{align*}
is effective. It follows that $[D]\in W_2$ and $W_2=\Pic(C_2, \mathbb{Q}.D_2)[2]\setminus \Pic(C_2)[2]$.
\end{proof}

\begin{lem}
\label{lemtheta}
Let $C$ be a nonsingular projective curve of genus $g$ over an algebraically closed field. Then every odd theta characteristic is an effective theta characteristic. Furthermore,
\begin{enumerate}
\item If $g\leq 3$, then $C$ has no vanishing theta characteristics.
\item If $g=4$ and $C$ is not hyperelliptic, then there exists a vanishing even theta characteristic if and only if $C$ is trigonal and there is a unique $g_3^1$ on $C$, in which case the unique vanishing even theta characteristic on $C$ is the $g_3^1$ on $C$.
\end{enumerate}
\end{lem}

\begin{proof}
If $[D]$ is a theta characteristic on $C$, then $\deg D=g-1$. The statement is trivial when $g=0$. When $g>0$ and $C$ is not rational, we have $l(D)\leq \deg D-1=g-2$, and so $l(D)\leq 1$ when $g\leq 3$. The first statement follows immediately. For the second statement see \cite[Th.~4.3]{Kulkarni17} and \cite[Section 2.8]{Vakil}.
\end{proof}


\subsection{Arguments using $3$-descent}

We assume throughout that $f(t)$ is squarefree.

\begin{thm}
\label{th3desc}
Let $n\in\mathbb{Q}, n>0$. Then there is a map
\begin{align*}
\phi_3:\left\{P=(x_0,y_0)\in E(\kbar(t))\mid y_0\in k[t], 0<\hat{h}(P)\leq n\right\}&\to W_{2n}(C_3)[3]\setminus \{0\}\\
P=(x_0,y_0)&\mapsto \left[\frac{1}{3}\divisor(y_0-y)\right]
\end{align*}
\begin{enumerate}
\item If $n<d/4$, the map $\phi_3$ is $3$-to-$1$ onto its image.
\item If $n\leq (d+1)/6$, the map $\phi_3$ is $3$-to-$1$ and surjective. \label{th3descii}
\end{enumerate}
\end{thm}

By Remark \ref{smallcan}, the only interesting values of $n$ in \ref{th3descii} are $n=d/6$ if $d\equiv 0,2,3,4\pmod{6}$ and $n=(d+1)/6$ if $d\equiv 5\pmod{6}$.

\begin{proof}
Let $P=(x_0(t),y_0(t))\in E(\kbar(t))$, with $y_0\in k[t]$ and $0<h(P)=\hat{h}(P)\leq n$.  It is elementary that $h(P)=\frac{1}{6}\max \{d, 2\deg y_0\}$.  By Proposition~\ref{prop:3descent}, we have $[D]\in \Pic(C_3,\mathbb{Q}.D_3)[3]$, where $D=\frac{1}{3}\divisor(y_0(t)-y)$.  By Lemma~\ref{RRlem-3},
\begin{align*}
\ord_\infty D=\frac{1}{3}\max \{d, 2\deg y_0\}=2h(P)\leq 2n=2n/\deg D_\infty & \text{ if }2\nmid d,\\
\max_{1\leq i\leq 2}\ord_{\infty_i} D=\frac{1}{3}\max \{d/2, \deg y_0\}=h(P)\leq 2n/\deg D_\infty & \text{ if }2\mid d.
\end{align*}

Thus, in any case, $D+2n/(\deg D_\infty)D_\infty$ is effective and $\phi_3$ is well-defined.   If $n<d/4$, then $P\in  E(\kbar[t])_{<d/4}$, and so $\phi_3$ is $3$-to-$1$ onto its image by the proof of Theorem~\ref{thm:main-3}~\ref{item:mainthsmallht-3}.

Suppose now that $n\leq (d+1)/6$.  We first claim that $\phi_3$ is $3$-to-$1$ onto its image. If $d\geq 3$, then this implies that $n<\frac{d}{4}$ and $\phi_3$ is $3$-to-$1$ from the above.   If $d=2$, then $0<h(P)\leq \frac{1}{2}$ easily implies that $h(P)=\frac{1}{3}$ and the claim follows by taking $n=\frac{1}{3}<\frac{d}{4}$.  If $d=1$ then $E(\kbar(t))$ is trivial and the claim is vacuously true. 

It remains to show that $\phi_3$ is surjective.   Let $[D]\in W_{2n}(C_3)[3]\setminus\{0\}$ so that 
\begin{align*}
D_0:=D+2n/(\deg D_\infty)D_\infty\in \Div(C_3,\mathbb{Q}.D_3) 
\end{align*}
is effective.  Then
\begin{align*}
3D=3D_0-\frac{6n}{\deg D_\infty}D_\infty
\end{align*}
is principal and $6n\leq d+1<d+2$.  It follows from Lemma~\ref{RRlem-3} that there exists a polynomial $y_0(t)\in k[t]$ with $\deg y_0\leq 3n$ and $a\in k$ such that $\divisor (y_0(t)+ay)=3D$.  If $a=0$, then this implies that $y_0(t)$ is a perfect cube (up to a constant) and $D$ itself is principal, contradicting $[D]\neq 0$.  Then $a\neq 0$ and replacing $y_0$ by $-ay_0$, we obtain a polynomial $y_0(t)\in k[t]$ with $\divisor (y_0-y)=3D$.  This implies that there exists $x_0\in \kbar[t]$ such that $y_0^2=x_0^3+f$.  Moreover, using Lemma~\ref{RRlem-3} to compute the poles of $y_0-y$, we find that $6n\geq d$. Let $P=(x_0,y_0)$.  Then $h(P)=\max\{\frac{d}{6}, \frac{1}{2}\deg x_0\}\leq n$ and $\phi_3(P)=[D]$. So $\phi_3$ is surjective as well.
\end{proof}

An elementary calculation using the Riemann-Hurwitz formula yields the canonical divisor (and the genus of $C_3$, already given in \eqref{eq:genusofC3}).

\begin{lem}
\label{canlem-3}
Let $K_{C_3}$ be the canonical divisor of $C_3$.  Then
\begin{align*}
K_{C_3}\sim 
\begin{cases}
(d-3)\infty \quad &\text{if $d$ is odd},\\
\left(\frac{d}{2}-2\right)D_\infty, \quad &\text{if $d$ is even}.
\end{cases}
\end{align*}
\end{lem}

\begin{lem}
\label{lemWC3}
Let
\begin{equation*}
g=g(C_3)= \left\{
\begin{array}{ll}
\frac{1}{2}(d-1) & \text{if $d$ is odd}\\
\frac{1}{2}(d-2)& \text{if $d$ is even}.
\end{array}\right.
\end{equation*}
\begin{enumerate}
\item
If $d$ is odd or $d\equiv 6\pmod{12}$ then
\begin{align*}
W_g(C_3)[3]&=\Pic(C_3)[3].
\end{align*}
\item If $d\equiv 2,4\pmod{6}$ then
\begin{align*}
W_{g+\frac{2}{3}}(C_3)[3]&=\Pic(C_3, \mathbb{Q}.D_3)[3] &&\text{if $g$ is even}\\
W_{g+\frac{1}{3}}(C_3)[3]&=\left(\Pic(C_3, \mathbb{Q}.D_3)[3]\setminus \Pic(C_3)[3]\right)\cup (\Pic(C_3)\cap W_{g-1}(C_3)[3])&&\text{if $g$ is odd}.
\end{align*}
\end{enumerate}
\end{lem}

\begin{proof}
Suppose first that $d$ is odd or $d\equiv 6\pmod{12}$. Then $\Pic(C_3, \mathbb{Q}.D_3)=\Pic(C_3)$.  For any line bundle $\mathcal{L}$, by the Riemann-Roch formula the condition $h^0(\mathcal{L}\otimes \O_{C_3}(\frac{g}{\deg D_{\infty}}D_\infty))>0$ is always satisfied, and so clearly $W_g(C_3)[3]=\Pic(C_3)[3]$ (note that when $d\equiv 6\pmod{12}$, $g$ is even and $g/\deg D_\infty\in\mathbb{Z}$).

Suppose now that $d\equiv 2,4\pmod{6}$. We first assume that $g$ is even.  Let $W_1=W_{g+\frac{2}{3}}(C_3)[3]\cap \Pic(C_3)$ and $W_2=W_{g+\frac{2}{3}}(C_3)[3]\cap \left(\Pic(C_3, \mathbb{Q}.D_3)\setminus \Pic(C_3)\right)$. In this case, since $\left\lfloor\frac{g+\frac{2}{3}}{\deg D_\infty}\right\rfloor=\left\lfloor\frac{g}{2}+\frac{1}{3}\right\rfloor=\frac{g}{2}=\frac{g}{\deg D_\infty}$, we see that for a $\mathbb{Z}$-divisor $D$, $D+\frac{g+\frac{2}{3}}{\deg D_\infty}D_\infty$ is effective if and only if $D+\frac{g}{\deg D_\infty}D_\infty$ is effective. By Riemann-Roch, if $\deg D=0$, then $h^0(D+\frac{g}{\deg D_\infty}D_\infty)>0$, and it follows that $W_1=\Pic(C_3)$. 

Now let $[D]\in \Pic(C_3, \mathbb{Q}.D_3)[3]\setminus \Pic(C_3)[3]$.  Then we may write the principal divisor $3D$ as $3D=3D'+\epsilon_1\infty_1+\epsilon_2\infty_2$, where $D'$ is a $\mathbb{Z}$-divisor and $\epsilon_i\in \mathbb{Z}$, $i=1,2$, not both divisible by $3$. Since $\deg 3D=0$, we find that $\epsilon_1\equiv -\epsilon_2\pmod{3}$. Then, after possibly interchanging the points at infinity, we can write $D=D'+\frac{1}{3}\infty_1-\frac{1}{3}\infty_2$ for some $\mathbb{Z}$-divisor $D'$. Note also that in this case $\infty_1$ and $\infty_2$ are $k$-rational. We also have
\begin{align*}
\deg\left(D+\frac{g+\frac{2}{3}}{\deg D_\infty}D_\infty-\frac{2}{3}\infty_1\right)=\deg\left(D'+\frac{g}{2}D_\infty\right)=g
\end{align*}
where $D'+\frac{g}{2}D_\infty$ is a $\mathbb{Z}$-divisor. Then by Riemann-Roch, $D+\frac{g+\frac{2}{3}}{\deg D_\infty}D_\infty-\frac{2}{3}\infty_1$ is linearly equivalent to an effective divisor, and $D\sim E-\frac{g+\frac{2}{3}}{\deg D_\infty}D_\infty$ for some effective $\mathbb{Q}$-divisor $E$. It follows that $[D]\in W_2$ and $W_2=\Pic(C_3, \mathbb{Q}.D_3)[3]\setminus \Pic(C_3)[3]$.

Now suppose that $g$ is odd. Similar to the previous case, we let $W_1=W_{g+\frac{1}{3}}(C_3)[3]\cap \Pic(C_3)$ and $W_2=W_{g+\frac{1}{3}}(C_3)[3]\cap \left(\Pic(C_3, \mathbb{Q}.D_3)\setminus \Pic(C_3)\right)$.  Since $\left\lfloor\frac{g+\frac{1}{3}}{\deg D_\infty}\right\rfloor=\left\lfloor\frac{g-1}{2}+\frac{2}{3}\right\rfloor=\frac{g-1}{2}=\frac{g-1}{\deg D_\infty}$, we see that for a $\mathbb{Z}$-divisor $D$, $D+\frac{g+\frac{1}{3}}{\deg D_\infty}D_\infty$ is effective if and only if $D+\frac{g-1}{\deg D_\infty}D_\infty$ is effective. It follows that $W_1=\Pic(C_3)\cap W_{g-1}(C_3)[3]$. Now let $[D]\in \Pic(C_3, \mathbb{Q}.D_3)[3]\setminus \Pic(C_3)[3]$.  Then as before, after possibly interchanging the points at infinity, we can write $D=D'-\frac{1}{3}\infty_1-\frac{2}{3}\infty_2$ for some $\mathbb{Z}$-divisor $D'$.  We also have
\begin{align*}
\deg\left(D+\frac{g+\frac{1}{3}}{\deg D_\infty}D_\infty-\frac{1}{3}\infty_1\right)=\deg\left(D'+\frac{g-1}{2}D_\infty\right)=g
\end{align*}
where $D'+\frac{g-1}{2}D_\infty$ is a $\mathbb{Z}$-divisor. Then by Riemann-Roch, $D+\frac{g+\frac{1}{3}}{\deg D_\infty}D_\infty-\frac{1}{3}\infty_1$ is linearly equivalent to an effective divisor, and $D\sim E-\frac{g+\frac{1}{3}}{\deg D_\infty}D_\infty$ for some effective $\mathbb{Q}$-divisor $E$. It follows that $[D]\in W_2$ and $W_2=\Pic(C_3, \mathbb{Q}.D_3)[3]\setminus \Pic(C_3)[3]$.

\end{proof}

\section{An application}

When $k=\kbar$ is algebraically closed, Theorems~\ref{th2desc}~\ref{th2descii} and \ref{th3desc}~\ref{th3descii} combined yield an unexpected relation between certain $2$-torsion and $3$-torsion sets associated to $C_2$ and $C_3$, respectively. This extends Table~\ref{table1} to higher genus curves, where however one lacks an explicit formula, and in general the quantities depend on the polynomial $f$ (and not just $\deg f$).

\begin{thm}
\label{Wbij}
Suppose that $k=\kbar$ is algebraically closed. Let $f\in \kbar[t]$ be a squarefree polynomial of positive degree $d$. Then
\begin{align*}
2\left(\#W_{\frac{d}{2}}(C_2)[2]-1\right)=3\left(\#W_{\frac{d}{3}}(C_3)[3]-1\right), \quad &&d\not\equiv 5\pmod{6}\\
2\left(\#W_{\frac{d+1}{2}}(C_2)[2]-1\right)=3\left(\#W_{\frac{d+1}{3}}(C_3)[3]-1\right), &&d\equiv 5\pmod{6}
\end{align*}
\end{thm}

The statement is trivial if $d\equiv 1\pmod{6}$ (see Remark \ref{smallcan}).

In general, it seems to be a difficult problem to bound the number of $n$-torsion points on a given subvariety $V$ of an abelian variety (in fact, if the subvariety $V$ does not contain a torsion translate of a nontrivial abelian subvariety, then the number of torsion points on $V$ is finite by a theorem of Raynaud (Manin-Mumford conjecture)). Using \eqref{eq:ratPicbound} and \eqref{eq:genusofC2}, we have the naive bound
\begin{align*}
\#W_{\frac{d}{2}}(C_2)[2]\leq 2^{2g(C_2)+1}\leq 2^{2d-1}.
\end{align*}

Taking advantage of Theorem~\ref{Wbij} immediately gives, for large $d$, an exponential improvement to the naive bound:
\begin{thm}
Let $f\in k[t]$ be a squarefree polynomial of positive degree $d$. Then
\begin{align*}
\#W_{\frac{d}{2}}(C_2)[2]\leq 
\begin{cases}
\frac{1}{2}(3^{d}-1) & \text{ if }d\equiv 2,3,4\pmod{6},\\
\frac{1}{2}(3^{d-1}-1) & \text{ if }d\equiv 0\pmod{6},
\end{cases}
\end{align*}
and
\begin{align*}
\#W_{\frac{d+1}{2}}(C_2)[2]\leq \frac{1}{2}(3^{d}-1) & \text{ if }d\equiv 5\pmod{6}.
\end{align*}
\end{thm}

When $d=2,3,5,6$, the inequality is easily seen to be sharp over an algebraically closed field (of characteristic not $2$ or $3$).

\begin{proof}
We may assume $k=\kbar$. If $d\equiv 2,4\pmod{6}$, then
\begin{align*}
\#W_{\frac{d}{2}}(C_2)[2]=\frac{3}{2}\#W_{\frac{d}{3}}(C_3)[3]-\frac{1}{2}\leq \frac{3}{2}3^{2g+1}-\frac{1}{2}\leq \frac{1}{2}(3^d-1).
\end{align*}
The other cases are similar.
\end{proof}

\section{Proof of Theorem~\ref{corSh2}}

Let $\Sigma$ be the set of places of bad reduction of $\Emin\to\mathbb{P}^1$ over $\kbar$. Then we have the root lattice $T=\oplus_{v\in R}T_v$, where $T_v$ is given by $0, A_2, D_4, E_6, E_8$ if the reduction type at $v$ is $\mathrm{II}, \mathrm{IV}, \mathrm{I}_0^*, \mathrm{IV}^*, \mathrm{II}^*$, respectively (see \cite{OS}). Taking the corresponding entries from \cite{OS}, we find Table~\ref{table2}.
\begin{table}
\caption{Oguiso-Shioda classification and Theorem~\ref{corSh2}}
\label{table2}
\centering
\begin{tabular}{|c|c|c|c|}
\hline
Type & No.~in \cite{OS} & $T$ & $E(\kbar(t))\cong$(Mordell-Weil Lattice $L$)$\oplus E(\kbar(t))_{\tors}$\\
\hline
$(1,0,0,0,1)$ & 62 & $E_8$ & $0$\\
\hline
$(0,1,0,1,0)$ & 69 & $A_2\oplus E_6$ & $\mathbb{Z}/3\mathbb{Z}$\\
\hline
$(2,0,0,1,0)$ & 27 & $E_6$ & $A_2^*$\\
\hline
$(0,0,2,0,0)$ & 73 & $D_4\oplus D_4$ & $\mathbb{Z}/2\mathbb{Z}\oplus \mathbb{Z}/2\mathbb{Z}$ \\
\hline
$(1,1,1,0,0)$ & 32 & $A_2\oplus D_4$ & $\frac{1}{6}\begin{pmatrix}
2 & 1  \\
1 & 2 
\end{pmatrix}$\\
\hline
$(3,0,1,0,0)$ & 9 & $D_4$ & $D_4^*$\\
\hline
$(0,3,0,0,0)$ & 39 & $A_2^3$ & $A_2^*\oplus \mathbb{Z}/3\mathbb{Z}$\\
\hline
$(2,2,0,0,0)$ & 11 & $A_2^2$ & $A_2^{*2}$\\
\hline
$(4,1,0,0,0)$ & 3 & $A_2$ & $E_6^*$\\
\hline
$(6,0,0,0,0)$ & 1 & $0$ & $E_8$\\
\hline

\end{tabular}
\end{table}

We now combine Table~\ref{table2} with the tools from the previous sections to prove Theorem~\ref{corSh2}.

\begin{proof}[Proof of Theorem~\ref{corSh2}]

The cases  of types $(1,0,0,0,1)$, $(0,1,0,1,0)$, and $(0,0,2,0,0)$ of Theorem~\ref{corSh2} are  immediate from Table~\ref{table2}, after noting the obvious torsion points.

Next, we consider the remaining type where $F$ is a perfect power: type $(0,3,0,0,0)$.  We prove the case where $f$ is the square of a cubic; the proof when $f$ is the square of a quadratic polynomial is similar. In this case, we have $f(t)=ac(t)^2$ for some monic separable cubic polynomial $c\in k[t]$ and $a\in k^*$. Write 
\begin{align*}
c(t)&=(t-\alpha_1)(t-\alpha_2)(t-\alpha_3), &&\alpha_i\in \kbar, i=1,2,3, \\
g(t)&=(\alpha_1-\alpha_2)(\alpha_1-\alpha_3)(t-\alpha_2)(t-\alpha_3)\in k(\alpha_1)[t],\\
\Delta&=(\alpha_1-\alpha_2)^2(\alpha_1-\alpha_3)^2(\alpha_2-\alpha_3)^2,
\end{align*}
where $\Delta\in k^*$ is the discriminant of $c(t)$.

We easily find the point
\begin{align*}
\left(\frac{2ag(t)}{\sqrt[3]{2a^2\Delta}},\sqrt{\frac{a}{\Delta}}g(t)((\alpha_2+\alpha_3-2\alpha_1)t+\alpha_1\alpha_2+\alpha_1\alpha_3-2\alpha_2\alpha_3)\right).
\end{align*}

Permuting the roots $\alpha_i$ and taking the six possible combinations of square and cube roots yields $18$ points of canonical height $\frac{1}{3}$ in $E(\kbar[t])$. Similarly, for the points $(x_0,y_0)$ above, the points
\begin{align*}
\lambda^t((-3x_0, 3\sqrt{-3}y_0))
\end{align*}
give $18$ points of canonical height $1$ in $E(\kbar[t])$. From the Mordell-Weil lattice in Table~\ref{table2}, these are all of the points of canonical height $\leq 1$, and the entry in Theorem~\ref{corSh2} follows easily from the explicit formulas for the points.

In the case $(2,2,0,0,0)$, we have $F=F_1F_2^2$ where $\deg F_1=\deg F_2=2$. Working over $\kbar$, we first compute the images of the descent maps on the points of canonical heights $1/3$ and $1$.  Using Lemma~\ref{lem:changeofcurve} and a change of variable, we reduce to considering the case $f=at^4+bt^3+ct^2\in \kbar[t]$, $a,c\in \kbar^*$, $b^2-4ac\neq 0$.   Then we find $12$ points of the form 
\begin{align*}
Q_i=(x_i,y_i)=\left(-\sqrt[3]{b+2\sqrt{ac}}t,\sqrt{a}t^2-\sqrt{c}t\right)
\end{align*}
for the various choices of the square and cube roots (choosing $\sqrt{ac}=\sqrt{a}\sqrt{c}$ consistent with the choices of $\sqrt{a}$ and $\sqrt{c}$), with each such point of canonical height $\frac{1}{3}$. It follows from the Mordell-Weil lattice entry in Table \ref{table2} that there are exactly $12$ points of the minimal canonical height $\mu=\frac{1}{3}$, and $12$ points of canonical height $1$. Thus, we have found all points of $E_{\min}$, and $\sqrt{-3}E_{\min}$ contains all points of canonical height $1$. The curve $C_2:x^3=-f(t)$ has genus $2$. Abusing notation, let $0$ and $\infty$ denote the unique points on $C_2$ mapping to $0$ and $\infty$, respectively, on $\mathbb{P}^1$ via the map $t:C_2\to \mathbb{P}^1$. Then $4\infty-2\cdot 0$ is a canonical divisor on $C_2$, and the map $Q_i\mapsto \frac{1}{2}\divisor(x_i-x)+(2\infty-0)$ is $2$-to-$1$ onto its image by Lemma~\ref{mu}. The images are effective theta characteristics, which for a genus $2$ curve are equivalent to odd theta characteristics. Since there are $6$ odd theta characteristics of $C$, the image must be precisely the odd theta characteristics of $C$. Similarly, let $t^{-1}(0)=\{P_1,P_2\}$ and $t^{-1}(\infty)=\{P_3,P_4\}$ in $C_3(\kbar)$, so that $D_3=P_1+P_2+P_3+P_4$. Then the map on $E_{\min}$, $Q_i\mapsto [\frac{1}{3}(y_i-y)]$, is $3$-to-$1$, and by explicit calculation the image is the set $[\frac{1}{3}(D_3+P_1+P_3),\frac{1}{3}(D_3+P_1+P_4), \frac{1}{3}(D_3+P_2+P_3), \frac{1}{3}(D_3+P_2+P_4)]\subset \Pic(C_3, \mathbb{Q}.D_3)[3]$.  Using Lemma~\ref{mu}, this proves the statement for points of canonical height $\frac{1}{3}$ and $1$. It remains to consider the points of canonical height $\frac{2}{3}$.

From the Mordell-Weil lattice, there are $36$ points of canonical height $2/3$ in $E(\kbar(t))$. By explicit calculation (the points $Q_i$ generate the Mordell-Weil group), the points of canonical height $2/3$ come in two types: there are $18$ such integral points $(x(t),y(t))$ with $(\deg x,\deg y)=(1,2)$ and $18$ integral points with $(\deg x, \deg y)=(2,3)$ and $x(0)=y(0)=0$. On each $\Gal(\kbar/k)$-invariant set, the map $Q\mapsto \frac{1}{2}\divisor(x_Q-x)+(2\infty-0)$ is $2$-to-$1$ onto the even theta characteristics, excluding the even theta characteristic $2\infty-0$ (there are $2^{2-1}(2^2+1)-1=9$ such even theta characteristics). The $3$-descent map on the points of canonical height $2/3$ maps onto the set $T=\{\pm[\frac{1}{3}(P_1-P_2)],\pm[\frac{1}{3}(P_3-P_4)]\}\subset \Pic(C_3, \mathbb{Q}.D_3)[3]$, and these maps induce a bijection between the set of points of canonical height $2/3$ and the product $T\times \{\text{even theta characteristics on $C_2$ excluding } [2\infty-0]\}$.  Then the corresponding entries in Theorem~\ref{corSh2} follow from the same reasoning as in the proof of Lemma~\ref{mu}.

In the case of type $(1,1,1,0,0)$, using Lemma~\ref{lem:changeofcurve} and a change of variable (over $k$), we may assume that $f$ has the form $f=at^2(t+1)$, $a\in k^*$. Then we find $6$ points of the form $(-\sqrt[3]{a}t, \sqrt{a}t)\in E(\kbar(t))$ (for the $6$ possible choices of the roots), each of canonical height $\frac{1}{6}$, and $6$ points of the form 
\begin{align*}
\lambda^t((3\sqrt[3]{a}t, 3\sqrt{-3a}t))=\left(-\frac{\sqrt[3]{a}(3t+4)}{3},\frac{\sqrt{-3a}(9t+8)}{9}\right),
\end{align*}
each of canonical height $\frac{1}{2}$. It follows from Table~\ref{table2} that these are all of the points of canonical heights $\frac{1}{6}$ and $\frac{1}{2}$. On the other hand, from \eqref{eq:ratPicbound}, $2(\#\Pic(C_2,\mathbb{Q}.D_2)[2]\setminus\{0\})\leq 6$ and $3(\#\Pic(C_3,\mathbb{Q}.D_3)[3]\setminus \{0\})\leq 6$, and this case of Theorem~\ref{corSh2} follows from Lemma~\ref{mu}. Note this case could also be stated in terms of the number of cube roots of $a$ in $k$ and the number of square roots of $a$ and $-3a$ in $k$.

The remaining cases of Theorem~\ref{corSh2} involve types $(2,0,0,1,0),(3,0,1,0,0),(4,1,0,0,0)$, and $(6,0,0,0,0)$. In each case, using Lemma~\ref{lem:changeofcurve} and a change of variables, we may assume that $f$ is squarefree, $2\leq \deg f\leq 6$. Except for the case of type $(3,0,1,0,0)$ and points of canonical height $1$, all of these statements follow from Theorem~\ref{th2desc}, Theorem~\ref{th3desc}, Lemma~\ref{lemWC2}, Lemma~\ref{lemWC3}, and Lemma~\ref{lemtheta}, with the following additional notes:
\begin{enumerate}
\item Type $(2,0,0,1,0)$, $\deg f=2$: From the Mordell-Weil lattice, there are $6$ points of canonical heights $\frac{1}{3}$ and $1$, and so $\sqrt{-3}E_{\min}$ is the set of points of canonical height $1$. Then  the statements for points of canonical height $1$ follow from Lemma~\ref{mu}.
\item Type $(3,0,1,0,0)$, $\deg f=3$:  We have $g(C_2)=g(C_3)=1$, and from Lemma~\ref{lemWC2}, $W_{\frac{3}{2}}(C_2)[2]\setminus \{0\}= \Pic(C_2, \mathbb{Q}.D_2)[2]\setminus \Pic(C_2)[2]$.
\item Type $(4,1,0,0,0)$, $\deg f=4$: We have $g(C_2)=3$ and $g(C_3)=1$. In this case, by Lemma~\ref{lemtheta}, the effective theta characteristics on $C_2$ coincide with the odd theta characteristics on $C_2$, and $\frac{1}{2}K_{C_2}$ is an odd theta characteristic. For points of canonical height $1$, we note that $W_3(C_2)[2]\setminus W_2(C_2)[2]$ can be identified with the even theta characteristics in $\Pic(C_2)$.  From Lemma~\ref{lemWC3}, we have $W_{\frac{4}{3}}(C_3)[3]\setminus \{0\}= \Pic(C_3, \mathbb{Q}.D_3)[3]\setminus \Pic(C_3)[3]$.
\item Type $(6,0,0,0,0)$, $\deg f=5,6$: We have $g(C_2)=4$ and $\frac{1}{2}K_{C_2}$ is a vanishing even theta characteristic (the unique vanishing even theta characteristic by Lemma~\ref{lemtheta}). So $W_3(C_2)[2]
\setminus\{0\}$ may be identified with the set of odd theta characteristics in $\Pic(C_2)$.
\end{enumerate}

Finally, we consider the case of type $(3,0,1,0,0)$ and canonical points of height $1$. It is possible to give an argument along the lines of the proofs of Theorem~\ref{th2desc} and Theorem~\ref{th3desc}, but instead we can simply appeal to the case $m=2$ of \cite[Th.~8.2(iv)]{Shioda} which states that in this case there are $24$ integral points in $E(\kbar[t])$ of canonical height $1$ of the form $(x(t),y(t))$, $\deg_t x(t)=2, \deg_t y(t)=3$.	By Theorem~\ref{thm:main-3}(ii) the map
\begin{align*}
\left\{P\in E(\kbar(t))\mid \hat{h}(P)=1\right\}\to \Pic(C_3)(\kbar)[3]\setminus\{0\}
\end{align*}
is $3$-to-$1$ onto its image. Since $g(C_3)=1$ and $\#\Pic(C_3)(\kbar)[3]-1=8$, the map is onto, with the preimage of a point a set of the form $\{P,\zeta_3 P,\zeta_3^2 P\}$. If $\hat{h}(P)=1$, then $P,\zeta_3 P, \zeta_3^2 P$ are easily seen to be inequivalent modulo $2E(\kbar(t))$, and from the form of the points the image of the descent map lands in $\Pic(C_2)(\kbar)[2]\setminus\{0\}\subset \Pic(C_2, \mathbb{Q}.D_2)(\kbar)[2]$.  As $g(C_2)=1$ and $\#\Pic(C_2)(\kbar)[2]-1=3$, it follows that we have a bijection
\begin{align*}
\left\{P\in E(\kbar(t))\mid \hat{h}(P)=1\right\}\to (\Pic(C_2)(\kbar)[2]\setminus\{0\})\times (\Pic(C_3)(\kbar)[3]\setminus\{0\})
\end{align*}
which induces a bijection
\begin{align*}
\left\{P\in E(k(t))\mid \hat{h}(P)=1\right\}\to (\Pic(C_2)[2]\setminus\{0\})\times (\Pic(C_3)[3]\setminus\{0\}).
\end{align*}
For $P\in E(\kbar(t))$, $\hat{h}(P)=1$, the image of the set $\{P,-P\}$ is of the form 
\begin{align*}
\{\rm{pt}\}\times \{\text{nontrivial cyclic subgroup of $\Pic(C_3)(\kbar)[3]$}\} \setminus\{0\}.
\end{align*}
It follows that
\begin{align*}
\#\{P\in E(\kbar(t))\mid x(P)\in k[t], \hat{h}(P)=1\}&=2(\#\Pic(C_2)[2]-1)\cdot\\&(\#\text{nontrivial $k$-rational cyclic subgroups of } \Pic(C_3)(\kbar)[3]).
\end{align*}
\end{proof}

\begin{rmq}
\label{rmk:Bremner}
We show in detail how to use our results to derive Theorem~1.2 of \cite{Bremner}, excluding the information on exceptional integral points and the explicit formulas for the integral points.

Suppose $f(t)\in\Q[t]$ is a squarefree cubic. Following Bremner \cite{Bremner}, for a cubic $g(t)=at^3+bt+c$, we write $f\sim g$ if there exists, $\alpha,d,e\in \mathbb{Q}, \alpha d\neq 0$, such that 
\begin{align*}
f(dt+e)=at^3+\alpha^4bt+\alpha^6c.
\end{align*}

Let $C_3$ be the elliptic curve over $\Q$ defined by $C_3:y^2=f(t)$. It is well-known that $C_3[3]\neq \{0\}$ if and only if $C_3$ has a Weierstrass model of the form $y^2=x^3+(6uv+27v^4)x+u^2-27v^6$ for some $u,v\in\mathbb{Z}$ (see, e.g., \cite{GT12}). Considering the cases $v=0$ and $v\neq 0$ separately and writing $\lambda=u/v^3$, equivalently, $C_3[3]\neq \{0\}$ if and only if $C_3$ has a Weierstrass model of the form $y^2=x^3+(6\lambda+27)x+\lambda^2-27$ or of the form $y^2=x^3+\lambda^2$ for some $\lambda\in\mathbb{Q}$.

Suppose now that we are in the first case of \eqref{Brem1}. Then since $a\in\mathbb{Q}^{*3}$, we have $f(t)\sim -t^3+bt+c$ for some $b,c\in\mathbb{Q}$. Then from the above (replacing $x$ by $-x$) we find that the first case of \eqref{Brem1} occurs if and only if
\begin{align*}
f(t)\sim -t^3-(6\lambda+27)t+\lambda^2-27
\end{align*}
or
\begin{align*}
f(t)\sim -t^3+\lambda^2
\end{align*}
for some $\lambda\in\mathbb{Q}$. Replacing $\lambda$ by $-(\lambda+5)$ in the first expression above, these parametrizations yield precisely the union of the cases (i)~a),~b) and (ii) ~a),~b) in \cite[Theorem~1.2]{Bremner} (depending on whether $2\lambda$ is a perfect cube), which are exactly the cases of \cite[Theorem~1.2]{Bremner} where there is a pair of points in $E(\Q[t])$ of canonical height $\frac{1}{2}$.

Now suppose that we are in the first case of \eqref{Brem2}. Then using $C_3[3]\neq \{0\}$, by the same argument (but no longer assuming $a$ is a perfect cube) we easily find
\begin{align*}
f(t)\sim at^3+a^3(6\lambda+27)t+a^4(\lambda^2-27)
\end{align*}
or
\begin{align*}
f(t)\sim at^3+a^4\lambda^2
\end{align*}
for some $a\in\mathbb{Q}^*$ and $\lambda\in\mathbb{Q}$. For the first case above, we compute twice the discriminant
\begin{align*}
2\Delta_{at^3+a^3(6\lambda+27)t+a^4(\lambda^2-27)}=-54a^{10}(\lambda+5)(\lambda+9)^3,
\end{align*}
which agrees with $2\Delta_f$ up to a $6$th power. Then in this case, $2\Delta_f$ is a perfect cube if and only if $2a(\lambda+5)=u^3$ for some $u\in\mathbb{Q}^*$. Solving for $a$ and substituting, we have
\begin{align*}
f(t)&\sim \frac{u^3}{2(\lambda+5)}t^3+\frac{u^9}{8(\lambda+5)^3}(6\lambda+27)t+\frac{u^{12}}{16(\lambda+5)^4}(\lambda^2-27)\\
&\sim \frac{1}{2(\lambda+5)}t^3+\frac{u^8}{8(\lambda+5)^3}(6\lambda+27)t+\frac{u^{12}}{16(\lambda+5)^4}(\lambda^2-27)\\
&\sim \frac{1}{2(\lambda+5)}t^3+\frac{1}{8(\lambda+5)^3}(6\lambda+27)t+\frac{1}{16(\lambda+5)^4}(\lambda^2-27)\\
&\sim  \lambda't^3-3\lambda'^2(\lambda'+1)t-\lambda'^2(8\lambda'^2-20\lambda'-1)/4
\end{align*}
where $\lambda'=-\frac{1}{2(\lambda+5)}$ (and we replaced $t$ by $-t$) in the last line. The parametrization in the last line is the union of the cases (i)~a) and (ii) ~d) in \cite[Theorem~1.2]{Bremner} (the case depending on whether $\lambda'$ is a perfect cube).

Finally, when $f(t)\sim at^3+a^4\lambda^2$, we similarly find that $2\Delta_f$ is a perfect cube if and only if $2a\lambda$ is a perfect cube. Solving for $a$ as above, by similar calculations we find that $f(t)\sim \lambda't^3+16(\lambda')^2$ for some $\lambda'\in\mathbb{Q}^*$, which is the union of the cases (i)~b) and (ii) ~g) in \cite[Theorem~1.2]{Bremner} (the case depending on whether $\lambda'$ is a perfect cube). Thus, in agreement with \cite[Theorem~1.2]{Bremner}, we find that there is a pair of points in $E(\Q[t])$ of canonical height $1$ precisely in cases (i)~a),~b) and (ii) ~d),~g) in \cite[Theorem~1.2]{Bremner}.
\end{rmq}



\bibliographystyle{amsalpha}
\bibliography{biblio3}

\providecommand{\bysame}{\leavevmode\hbox to3em{\hrulefill}\thinspace}
\providecommand{\MR}{\relax\ifhmode\unskip\space\fi MR }
\providecommand{\MRhref}[2]{%
  \href{http://www.ams.org/mathscinet-getitem?mr=#1}{#2}
}
\providecommand{\href}[2]{#2}
\begin{thebibliography}{ALRM07}

\bibitem[ALRM07]{ALRM07}
Scott Arms, \'{A}lvaro Lozano-Robledo, and Steven~J. Miller, \emph{Constructing
  one-parameter families of elliptic curves with moderate rank}, J. Number
  Theory \textbf{123} (2007), no.~2, 388--402. \MR{2301221}

\bibitem[BK77]{BK1977}
Armand Brumer and Kenneth Kramer, \emph{The rank of elliptic curves}, Duke
  Math. J. \textbf{44} (1977), no.~4, 715--743. \MR{457453}

\bibitem[Bre91]{Bremner}
Andrew Bremner, \emph{Some simple elliptic surfaces of genus zero}, Manuscripta
  Math. \textbf{73} (1991), no.~1, 5--37. \MR{1124308}

\bibitem[Cle69]{Clebsch1869}
Alfred Clebsch, \emph{zur theorie der bin\"aren formen sechster ordnung und zur
  dreitheilung der hyperelliptischen functionen}, Abhandlungen der
  K\"oniglichen Gesellschaft der Wissenschaften in G\"ottingen \textbf{14}
  (1869), 17--76.

\bibitem[Dav65]{Dav}
Harold Davenport, \emph{On {$f^{3}\,(t)-g^{2}\,(t)$}}, Norske Vid. Selsk. Forh.
  (Trondheim) \textbf{38} (1965), 86--87. \MR{0186621}

\bibitem[Elk99]{Elkies}
Noam~D. Elkies, \emph{The identification of three moduli spaces},
  \texttt{arXiv:math/9905195 [math.AG]}, 1999.

\bibitem[Gil09]{gillibert09}
Jean Gillibert, \emph{Prolongement de biextensions et accouplements en
  cohomologie log plate}, Int. Math. Res. Not. IMRN (2009), no.~18, 3417--3444.
  \MR{2535005}

\bibitem[GL22]{gl18a}
Jean Gillibert and Aaron Levin, \emph{Descent on elliptic surfaces and
  arithmetic bounds for the {M}ordell-{W}eil rank}, Algebra Number Theory
  \textbf{16} (2022), no.~2, 311--333. \MR{4412575}

\bibitem[GST12]{GT12}
Irene Garc\'{\i}a-Selfa and Jos\'{e}~M. Tornero, \emph{A complete {D}iophantine
  characterization of the rational torsion of an elliptic curve}, Acta Math.
  Sin. (Engl. Ser.) \textbf{28} (2012), no.~1, 83--96. \MR{2863752}

\bibitem[HS88]{HS}
Marc Hindry and Joseph~H. Silverman, \emph{The canonical height and integral
  points on elliptic curves}, Invent. Math. \textbf{93} (1988), no.~2,
  419--450. \MR{948108}

\bibitem[Kul17]{Kulkarni17}
Avinash Kulkarni, \emph{An arithmetic invariant theory of curves from {$E_8$}},
  2017.

\bibitem[Lan60]{lang1960}
Serge Lang, \emph{Integral points on curves}, Inst. Hautes \'{E}tudes Sci.
  Publ. Math. (1960), no.~6, 27--43. \MR{130219}

\bibitem[Lev14]{levin14}
Aaron Levin, \emph{Linear forms in logarithms and integral points on
  higher-dimensional varieties}, Algebra Number Theory \textbf{8} (2014),
  no.~3, 647--687. \MR{3218805}

\bibitem[Mas84]{mason84}
R.~C. Mason, \emph{Diophantine equations over function fields}, London
  Mathematical Society Lecture Note Series, vol.~96, Cambridge University
  Press, Cambridge, 1984. \MR{754559}

\bibitem[OS91]{OS}
Keiji Oguiso and Tetsuji Shioda, \emph{The {M}ordell-{W}eil lattice of a
  rational elliptic surface}, Comment. Math. Univ. St. Paul. \textbf{40}
  (1991), no.~1, 83--99. \MR{1104782}

\bibitem[Pil45]{Pillai}
S.~S. Pillai, \emph{On the equation {$2^x-3^y=2^X+3^Y$}}, Bull. Calcutta Math.
  Soc. \textbf{37} (1945), 15--20. \MR{0013386}

\bibitem[Shi90]{Shioda2}
Tetsuji Shioda, \emph{On the {M}ordell-{W}eil lattices}, Comment. Math. Univ.
  St. Paul. \textbf{39} (1990), no.~2, 211--240. \MR{1081832}

\bibitem[Shi05]{Shioda}
\bysame, \emph{Elliptic surfaces and {D}avenport-{S}tothers triples}, Comment.
  Math. Univ. St. Pauli \textbf{54} (2005), no.~1, 49--68. \MR{2153955}

\bibitem[SS08]{SS}
Matthias Sch\"utt and Andreas Schweizer, \emph{On {D}avenport-{S}tothers
  inequalities and elliptic surfaces in positive characteristic}, Q. J. Math.
  \textbf{59} (2008), no.~4, 499--522. \MR{2461271}

\bibitem[ST82]{ST}
R.~J. Stroeker and R.~Tijdeman, \emph{Diophantine equations}, Computational
  methods in number theory, {P}art {II}, Math. Centre Tracts, vol. 155, Math.
  Centrum, Amsterdam, 1982, pp.~321--369. \MR{702521}

\bibitem[Tat75]{tate75}
John Tate, \emph{Algorithm for determining the type of a singular fiber in an
  elliptic pencil}, Modular functions of one variable, {IV} ({P}roc.
  {I}nternat. {S}ummer {S}chool, {U}niv. {A}ntwerp, {A}ntwerp, 1972), Springer,
  Berlin, 1975, pp.~33--52. Lecture Notes in Math., Vol. 476. \MR{0393039}

\bibitem[Vak01]{Vakil}
Ravi Vakil, \emph{Twelve points on the projective line, branched covers, and
  rational elliptic fibrations}, Math. Ann. \textbf{320} (2001), no.~1, 33--54.
  \MR{1835061}

\bibitem[vG92]{vGeemen1992}
Bert van Geemen, \emph{Projective models of {P}icard modular varieties},
  Classification of irregular varieties ({T}rento, 1990), Lecture Notes in
  Math., vol. 1515, Springer, Berlin, 1992, pp.~68--99. \MR{1180338}

\bibitem[\v{S}63]{shafarevich62}
I.~R. \v{S}afarevi\v{c}, \emph{Algebraic number fields}, Proc. {I}nternat.
  {C}ongr. {M}athematicians ({S}tockholm, 1962), Inst. Mittag-Leffler,
  Djursholm, 1963, pp.~163--176. \MR{0202709}

\end{thebibliography}



\bigskip

\textsc{Jean Gillibert}, Universit\'e de Toulouse, Institut de Math{\'e}matiques de Toulouse, CNRS UMR 5219, 118 route de Narbonne, 31062 Toulouse Cedex 9, France.

\emph{E-mail address:} \texttt{jean.gillibert@math.univ-toulouse.fr}
\medskip

\textsc{Emmanuel Hallouin}, Universit\'e de Toulouse, Institut de Math{\'e}matiques de Toulouse, CNRS UMR 5219, 118 route de Narbonne, 31062 Toulouse Cedex 9, France.

\emph{E-mail address:} \texttt{hallouin@univ-tlse2.fr}
\medskip

\textsc{Aaron Levin}, Department of Mathematics, Michigan State University, 619 Red Cedar Road, East Lansing, MI 48824.

\emph{E-mail address:} \texttt{adlevin@math.msu.edu}


\end{document}